\documentclass[10pt]{amsart}

\usepackage{mathtools}
\usepackage[round]{natbib}
\setcitestyle{numbers,square}
\usepackage{hyperref}

\usepackage{accents}

\newcommand{\diff}{\,\mathrm{d}}
\newcommand{\diffns}{\mathrm{d}}
\usepackage[dvipsnames]{xcolor}

\newtheorem{defi}{Definition}[section]
\newtheorem{thm}[defi]{Theorem}
\newtheorem{lemm}[defi]{Lemma}

\newcommand{\hata}{\hat\alpha}

\newcommand{\PP}{\mathbb{P}}
\newcommand{\E}{\mathbb{E}}

\DeclareMathOperator{\sgn}{sgn}

\title[Maximum principle of SDEs with measurable drifts]{Maximum principle for stochastic control of SDEs with measurable drifts}
\author{Olivier Menoukeu-Pamen \& Ludovic Tangpi}

\thanks{University of Liverpool, Princeton University; \\  
menoukeu@liverpool.ac.uk, ludovic.tangpi@princeton.edu.}	
\date{\today}

\begin{document}

\begin{abstract} 
	In this paper, we consider stochastic optimal control of systems driven by stochastic differential equations with irregular drift coefficient.
	We establish a necessary and sufficient stochastic maximum principle. To achieve this, we first derive an explicit representation of the first variation process (in Sobolev sense ) of the controlled diffusion. 
	Since the drift coefficient is not smooth, the representation is given in terms of the local time of the state process. Then we construct a sequence of optimal control problems with smooth coefficients by an approximation argument. Finally, we use Ekeland's variational principle to obtain an approximating adjoint process from which we derive the maximum principle by passing to the limit. 
\end{abstract}

\maketitle

\section{Introduction}
	
Let $T \in (0,\infty)$ be a given deterministic time horizon and $d \in \mathbb{N}$, let $\Omega := C([0,T],\mathbb{R}^d)$ be the canonical space of continuous paths.
We denote by $B$ the canonical process and by $\PP$ the Wiener measure.
Equip $\Omega$ with $(\mathcal{F}_t)_{t\in [0,T]}$, the $\PP$-completion of the canonical filtration of $B$.
Given a $d$-dimensional vector $\sigma$ and a function $b: [0,T]\times \mathbb{R}\times \mathbb{R}^m\to \mathbb{R}$, we consider a controlled diffusion of the form
\begin{align}\label{eqSpro1}
	\diff X^\alpha (t)  =  b(t,X^\alpha(t),\alpha(t))\diff t +\sigma\diff B(t) ,\quad t \in [ 0,T] ,\quad X^\alpha(0) = x_0
\end{align}
and the control problem
\begin{align} \label{eqconpb1}
	V(x_0) := \sup_{\alpha \in \mathcal{A}}J(\alpha).
\end{align}
Hereby, the performance functional $J$ is given by
\begin{align*}
	J(\alpha):= \E\Big[ \int_0^T  f(s,X^\alpha(s),\alpha(s))\diff s + g(X^\alpha(T))\Big],\label{perfunct1}
\end{align*}
where, $f$ and $g$ may be seen as profit and bequest functions, respectively. 
The set $\mathcal{A}$ is the set of admissible controls and is defined as the set of progressively measurable processes $\alpha$ valued in a closed convex set $\mathbb{A}\subseteq \mathbb{R}^m$ such that \eqref{eqSpro1} admits a unique strong solution.
The goal of the present article is to derive the maximum principle for the above control problem when the drift $b$ is merely measurable is the state variable $x$.
	
The stochastic maximum principle is arguably one of the most prominent ways to tackle stochastic control problems as \eqref{eqconpb1} by fully probabilistic methods.
It is the direct generalization to the stochastic framework of the maximum principle of Pontryagin \cite{PontryaginBook} in deterministic control.
It gives a necessary condition of optimality in the form of a two-point boundary value problem and a maximum condition on the Hamiltonian.
More precisely let the Hamiltonian $H$ be defined as
$$H(t, x, y, a): = f(t,x,a) + b(t,x,a)y$$
and assume just for a moment the functions $b,f$ and $g$ to be continuously differentiable.
Then, if $\hat\alpha \in \mathcal{A}$ is an optimal control, then according to the stochastic maximum principle, it holds $H(t, X^{\hat\alpha}(t), Y(t), \hat\alpha(t)) \ge H(t, X^{\hat\alpha}(t), Y(t), a) $ $P\otimes dt$-a.s. for every $a \in \mathbb{A}$ where $(Y,Z)$ are adapted processes solving the so-called adjoint equation
\begin{equation*}
	dY(t) = - \partial_xf(t, X^{\hat\alpha}(t),\hat\alpha(t)) - \partial_xb(t, X^{\hat\alpha}(t),\hat\alpha(t))Y(t)\diff t + Z(t)\diff B(t),\quad Y(T) = \partial_xg(X^{\hat\alpha}(T)).
\end{equation*} 
Under additional convexity conditions, this necessary condition is sufficient.
The interest of the maximum principle is that it reduces the solvability of the control problem \eqref{eqconpb1} to that of a (scalar) variational problem, and therefore allows to derive (sometimes explicit) characterizations of optimal controls.
We refer for instance to \cite{MR3629171,YZ99} for proofs and historical remarks.
The maximum principle has far-reaching consequences and is widely used in the stochastic control and stochastic differential game literature \cite{Car-Del15,MR3325083,Pen90,Pontryagin,optimierung}. 
Its use also fueled by recent progress on the theory of forward backward SDEs.
We refer the reader for instance to, \cite{Delarue,FbsdeRough,Ma-Zhang11,Zhang_Book17,dqFBSDE} and the references therein.
	
The maximum principle roughly presented above naturally requires differentiability of the coefficients of the control problem, which precludes the applicability of this method to control problems with non-smooth coefficients.
The effort to extend the stochastic maximum principle to problems with non-smooth coefficients started with the work of Merzedi \cite{Mer88} who derived a necessary condition of optimality for a problem with a Lipschitz continuous drift, but not necessarily differentiable everywhere in the state and the control variable.
His result was further extended, notably to degenerate diffusion cases and singular control problems in \cite{Bah-Dje-Mer07,Bah-Dje-Mer-AMO07,Bah-Chi-Dje-Mer}. See also \cite{KoMe15} for the infinity horizon case.
	
The present work considers the case where $b$ is Borel measurable in $x$ and bounded, and we will derive both necessary and the sufficient conditions of optimality. 
At this point, an immediate natural question is: What form should the adjoint equation take in this case?
The starting point of our argument is the following simple observation:
When $b$ is differentiable, the adjoint equation is explicitly solvable, with the solution given by
\begin{equation*}
	Y(t) = \E\Big[\Phi^{\hat\alpha}(t,T)\partial_xg(X^{\hat\alpha}(T)) + \int_t^T\Phi^{\hat\alpha}(t,s)\partial_xf(s, X^{\hat\alpha}(s),\hat\alpha(s))\diff s\mid \mathcal{F}_t \Big],
\end{equation*}
where the process 
\begin{equation}
\label{eq:flow.smooth}
	\Phi^{\hat\alpha}(t,s) = e^{\int_t^s\partial_xb(u, X^{\hat\alpha}(u),\hat\alpha(u))\diff u}\quad 0\le t\le s\le T
\end{equation} 
is the first variation process (in the Sobolev sense) of the dynamical system $X^{\hat\alpha,x}$ solving \eqref{eqSpro1} with initial condition $X^{\hat\alpha,x}_0 = x$.
This suggests the form of the adjoint process when $b$ is not differentiable, since it is well-known that despite the roughness of the drift $b$, the dynamical system $X^{\hat\alpha,x}$ is still differentiable (at least in the Sobolev sense), due to Brownian regularization \cite{MNP2015} and therefore admits a flow.
The crux of our argument will be to make use of this Sobolev differential stochastic flow to define the \emph{adjoint process} (rather than the adjoint equation) in the non-smooth case to prove necessary and sufficient conditions of optimality.
	
Throughout this work the functions $f$ and $g$ are assumed to be continuously differentiable with bounded first derivatives.
In particular, we will assume
\begin{equation*}
	\sigma \in \mathbb{R}^d \text{ satisfies } |\sigma|^2>0 \quad \text{and}\quad	|f(t, x, a)| + |g(x)| \le C(1 + |x|)\quad \text{for all $(t,x,a)$ and some $C>0$.}
\end{equation*}
The main results of this work are the following necessary and sufficient conditions in the Pontryagin stochastic maximum principle.
\begin{thm}
\label{thm:necc}
	Assume that $b$ satisfies $b(t,x,a):= b_1(t,x) + b_2(t,x,a)$ where $b_1$ is a bounded, Borel measurable function and $b_2$ is bounded measurable, and continuously differentiable in its second and third variables with bounded derivatives.
	Let $\hat\alpha \in \mathcal{A}$ be an optimal control and let $X^{\hat\alpha}$ be the associated optimal trajectory.
	Then the flow $\Phi^{\hat\alpha}$ of $X^{\hat\alpha}$ is well-defined and it holds 
	\begin{equation}
	\label{eq:nec.cond}
		\partial_{\alpha}H(t, X^{\hat\alpha}(t),Y^{\hat\alpha}(t),\hat\alpha(t) )\cdot(\beta - \hat\alpha(t)) \ge 0 \quad \PP\otimes \diff t\text{-a.s. for all } \beta \in \mathcal{A},
	\end{equation} 
	where $Y^{\hat\alpha}$ is the adjoint process given by
	\begin{equation}
	\label{eq:adj.proc}
		Y^{\hat\alpha}(t) := \E\Big[\Phi^{\hat\alpha}(t, T) g_x( X^{\hat\alpha}(T)) + \int_t^T\Phi^{\hat\alpha}(t,s) f_x(s, X^{\hat\alpha}(s), \hat\alpha(s))\mathrm{d}s\mid \mathcal{F}_t \Big].
	\end{equation}	
\end{thm}
	
\begin{thm}
\label{thm:suff}
	Let the conditions of Theorem \ref{thm:necc} be satisfied, further assume that $g$ and $(x,a)\mapsto H(t,x,y,a)$ are concave.
	Let $\hat\alpha\in \mathbb{A}$ satisfy
	\begin{equation}
	\label{eq:suff.con}
		\partial_\alpha H(t, X^{\hat\alpha}(t), Y^{\hat\alpha}(t), \hat\alpha_t)=0 \quad \PP\otimes \diff t\text{-a.s.}
	\end{equation}
	with $Y$ given by \eqref{eq:adj.proc}.
	Then, $\hat\alpha$ is an optimal control.
\end{thm}
We will elaborate on the conditions imposed in the above theorems in section \ref{subsec.conclusion}.
Let us at this point remark that these results correspond exactly to the classical version of the stochastic maximum principle when $b$ is smooth.
The only difference here being the fact that the process $\Phi^{\hat\alpha}$ seems abstract, as it is obtained from an existence result (of the flow). 
As noted by \cite{BMBPD17}, it turns out that when the drift is not smooth, the flow $\Phi^{\hat\alpha}$ still admits an explicit representation much similar to \eqref{eq:flow.smooth}.
This representation will be extended to the present controlled case (see Theorem \ref{Thmexpliflowder}) and will be used in the proof of the maximum principle.

The remainder of the article is dedicated to the proofs of Theorem \ref{thm:necc} and \ref{thm:suff}. 
The necessary condition is proved in the next section and the sufficient condition is proved in section \ref{sec:sufficient}.
The paper ends with an appendix on explicit representations of the flow of SDEs with measurable and random drifts.

	\section{The necessary condition for optimality}
	\label{sec:neccessary}
	
The goal of this section is to prove Theorem \ref{thm:necc}.
Let us first precise the definition of the set of admissible controls.
Let $\mathbb{A}\subseteq \mathbb{R}^m$ be a closed convex subset of $\mathbb{R}^m$.
The set of admissible controls is defined as:
\begin{multline*}
	\mathcal{A} := \Big\{\alpha:[0,T]\times \Omega\to \mathbb{A}, \text{ progressive, \eqref{eqSpro1} has a unique strong solution and }\\ \E\big[\sup_{t\in[0,T]}|\alpha(t)|^2\big]< M \Big\}
\end{multline*}
for some $M>0$.
The difficulty in the existence and uniqueness of \eqref{eqSpro1} is the fact that the drift $b$ is both non-smooth and depends on the random term $\alpha$.
Such equations were treated in \cite{MenTan19}.
In fact, consider the set $\mathcal{A}'$ defined as: 	
The set of progressively measurable processes $\alpha:[0,T]\times \Omega\to \mathbb{A}$ which are Malliavin differentiable (with Malliavin derivative $D_s\alpha(t)$), with
\begin{equation}\label{eqcondal1}
	\E\Big[\int_0^T|\alpha(t)|^2\diff t \Big] + \sup_{s\in [0,T]}\E\Big[\Big(\int_0^T|D_s\alpha(t)|^2\diff t\Big)^4 \Big] <\infty
\end{equation}
and such that there are constants $C,\eta>0$ (possibly depending on $\alpha$) such that
\begin{equation}\label{eqcondal2}
	\E[|D_s \alpha(t) - D_{s'}\alpha(t)|^4] \le C|s-s'|^\eta.
\end{equation}
It follows from \cite[Theorem 1.2]{MenTan19} that if the drift satisfies the conditions of Theorem \ref{thm:necc}, then the SDE \eqref{eqSpro1} is uniquely solvable for every $\alpha \in \mathcal{A}'$.
Since we do not make use of Malliavin differentiability in the present article we restrict ourselves to the set of admissible controls $\mathcal{A}$.
For later reference, note that for every $\alpha\in \mathcal{A}$ it holds $E[\sup_{t\in [0,T]}|X^\alpha(t)|^p]<\infty$ for every $p\ge1$.
		
In the rest of the article, we let $b_n$ be a sequence of functions defined by $b_n:= b_{1,n} + b_{2}$ such that $b_{1,n}: [0, T ] \times  \mathbb{R}  \rightarrow  \mathbb{R}, n \geq  1$ are smooth functions with compact support and converging a.e. to $b_1$. 
Since $b_1$ is bounded, the sequence $b_{1,n}$ can also be taken bounded.
We denote by $X^{\alpha}_n$ the solution of the SDE \eqref{eqSpro1} with drift $b$ replaced by $b_n$.
This process is clearly well-defined since $b_n$ is a Lipschitz continuous function.
Similarly, we denote respectively by $J_n$ and $V_n$ the performance and the value function of the problem when the drift $b$ is replaced by $b_n$.
That is, we put
\begin{equation*}
	J_n(\alpha) := \E\Big[\int_0^Tf(s, X^\alpha_n(s), \alpha(s))\diff s + g(X_n^\alpha(T)) \Big],\quad V_n(x_0) := \sup_{\alpha \in \mathcal{A}}J_n(\alpha)
\end{equation*}
and 
\begin{equation*}
	\diff X_n^\alpha (t)  =  b_n(t,X_n^\alpha(t),\alpha(t))\diff t +\sigma\diff B(t) ,\quad t \in [ 0,T] ,\quad X^\alpha(0) = x_0.
\end{equation*}
Furthermore, we denote by $\delta$ the distance
\begin{equation*}
	\delta(\alpha_1, \alpha_2) : = \E\big[\sup_{t \in[0,T]}|\alpha_1(t) - \alpha_2(t)|^{2} \big]^{1/2} .
\end{equation*}
	
The general idea of the proof will be to start by showing that an optimal control for the problem \eqref{eqconpb1} is also optimal for an appropriate perturbation of the approximating problem with value $V_n(x_0)$.
This is due to the celebrated variational principle of Ekeland.
This maximum principle for control problems with smooth drifts will involve the state process $X_n^{\hat\alpha_n}$ and its flow $\Phi^{\hat\alpha_n}_n$.
The last and most demanding step is to pass to the limit and show some form of "stability" of the maximum principle.
We first address this limit step by a few intermediary technical lemmas that will be brought together to prove Theorem \ref{thm:suff} at the end of this section.
\begin{lemm}
\label{lem:conv.Xnn}
	We have the following bounds:
	\begin{itemize}
		\item[(i)] For every $\alpha_1,\alpha_2 \in \mathcal{A}$ it holds that
			$$ 
				\E\big[| X^{\alpha_1}_n(t) - X^{\alpha_2}(t)| \big] \le C\Big( \delta(\alpha_1,\alpha_2) + \Big(\int_0^T\frac{1}{\sqrt{2\pi s}}e^{\frac{|x_0|^2}{2s}}\int_{\mathbb{R}^d}\big|b_{1,n} (s,\sigma y)-b_{1} (s,\sigma y)\big|^4e^{-\frac{|y|^2}{4s}}\diffns y\diff s\Big)^{1/2}\Big).
			$$ 
			\item[(ii)]Given $k \in \mathbb{N}$, for every sequence $(\alpha_n)$ in $\mathcal{A}$ converging to some $\alpha\in \mathcal{A}$ it holds that 
			$$ \E\big[| X^{\alpha_n}_k(t) - X^{\alpha}_k(t)|^2 \big] \to 0.$$
	\end{itemize}
\end{lemm}
\begin{proof}
	Adding and subtracting the same term and then using the fundamental theorem of calculus, we arrive at
	\begin{align*}
		&X_n^{\alpha_1}(t) - X^{\alpha_2}(t) 
		= \int_0^t\int_0^1\partial_xb_{1,n}(s,  \Lambda_n(\lambda,s)) + \partial_xb_2\big(s,\Lambda_n(\lambda,s),\alpha_1(s)\big)\mathrm{d}\lambda(X^{\alpha_1}_n(s) - X^{\alpha_2}(s))\mathrm{d}s\\
		& + \int_0^t b_{1,n}(s, X^{\alpha_2}(s)) - b_1(s, X^{\alpha_2}(s))\mathrm{d}s + \int_0^tb_2(s, X^{\alpha_2}(s),\alpha_1(s)) - b_2(s, X^{\alpha_2}(s), \alpha_2(s))\diff s,
	\end{align*}
	where $\Lambda_n(\lambda,t)$ is the process given by $\Lambda_n(\lambda,t):= \lambda X^{\alpha_1}_n(t) + (1 - \lambda)X^{\alpha_2}(t)$.
	Therefore, we obtain that $X^{\alpha_1}_n - X^{\alpha_2}$ admits the representation
	\begin{align*}
		&X^{\alpha_2}(t) - X_n^{\alpha_2}(t) = \int_0^t\exp\Big(\int_{s}^t\int_0^1\partial_xb_{1,n}(r,  \Lambda_n(\lambda,r)) + \partial_xb_2(r,\Lambda_n(\lambda,r), \alpha_1(r))\mathrm{d}\lambda\mathrm{d}r \Big)\\
		&\times \Big(b_{1,n}(s, X^{\alpha_2}(s)) - b_1(s, X^{\alpha_2}(s)) +b_2(s, X^{\alpha_2}(s),\alpha_1(s)) - b_2(s, X^{\alpha_2}(s),{\alpha_2}(s))\Big)\mathrm{d}s.
	\end{align*}
	Hence, taking expectation on both sides above and then using twice Cauchy-Schwarz inequality, we have that
	\begin{align}
	\notag
		&\E\big[|X^{\alpha_1}_n(t) - X^{\alpha_1}(t)|\big] \le \E\Big[\int_0^t \exp\Big(2\int_{s}^t\int_0^1\partial_xb_{1,n}(r,  \Lambda_n(\lambda,r)) + \partial_xb_2(r,\Lambda_n(\lambda,r),\alpha_1(r))\mathrm{d}\lambda\mathrm{d}r \Big)\mathrm{d} s\Big]^{1/2}\\
		\label{eq:estim.diff.x}
		&\times \E\Big[\int_0^{t}|b_1(s, X^{\alpha_2}(s)) - b_{1,n}(s, X^{\alpha_2}(s))|^2 + |b_2(s, X^{\alpha_2}(s),\alpha_1(s)) - b_2(s, X^{\alpha_2}(s),{\alpha_2}(s))|^2\diff s\Big]^{1/2}.
	\end{align}
	By Lipschitz continuity of $b_2$, the last term on the right hand side is estimated as
	\begin{equation}
		\label{eq:estim.alpha12}
		\E\Big[\int_0^T|b_2(s, X^{\alpha_2}(s),\alpha_1(s)) - b_2(s, X^{\alpha_2}(s),{\alpha_2}(s))|^2\diff s \Big] \le C\E\Big[\int_0^T|\alpha_1(s) - \alpha_2(s)|^2\diff s \Big]\le C(\delta(\alpha_1,\alpha_2))^{2}.
	\end{equation}
	Moreover, denoting $$\mathcal{E}\Big(\int_0^Tq(s)\diff B(s) \Big) = \exp\Big(\int_0^Tq(s)\diff B(s) - \frac12\int_0^T|q(s)|^2\diff s \Big),$$ the second integral on the right hand side of \eqref{eq:estim.diff.x} can be further estimated as follows:
	\begin{align*}
		&\E\Big[\int_0^{T}|b_1(s, X^{\alpha_2}(s)) - b_{1,n}(s, X^{\alpha_2}(s))|^2\mathrm{d}s\Big]\\ 
		& = \E\Big[\mathcal{E}\Big(\frac{\sigma^\top}{|\sigma|^2}\int_0^Tb(s, X^{\alpha_2}(s),{\alpha_2}(s))\mathrm{d}B(s) \Big)^{1/2}\mathcal{E}\Big(\int_0^T\frac{\sigma^\top}{|\sigma|^2}b(s, X^{\alpha_2}(s),{\alpha_2}(s))\mathrm{d}B(s) \Big)^{-1/2}\\
		&\quad \times\int_0^{T}|b_1(s, X^{\alpha_2}(s)) - b_{1,n}(s, X^{\alpha_2}(s))|^2\mathrm{d}s \Big]\\
		& \le C\E_{\mathbb{Q}}\Big[ \int_0^{T}|b_1(s, X^{\alpha_2}(s)) - b_{1,n}(s, X^{\alpha_2}(s))|^4\mathrm{d}t \Big]^{1/2}
	\end{align*}
	for some constant $C>0$ and the probability measure $\mathbb{Q}$ is the measure with density 
	\begin{equation}
	\label{eq:def.probab.Q}
		\frac{\diff \mathbb{Q}}{\diff \PP} := \mathcal{E}\Big(\int_0^T\frac{\sigma^\top}{|\sigma|^2}b(s, X^{\alpha_2}(s),{\alpha_2}(s))\mathrm{d}B(s) \Big).
	\end{equation}
	Note that we used Cauchy-Schwarz inequality and then the fact that $b$ is bounded to get $\E[(\frac{\diff \mathbb{Q}}{\diff\PP})^{-1}]\le C$.
	By Girsanov's theorem, under the measure $\mathbb{Q}$, the process $(X^{\alpha_2}(t) - x_0)\sigma^\top/|\sigma|^2 $ is a Brownian motion.
	Thus, it follows that
	\begin{align}
		\notag
		\E_{\mathcal{Q}}\Big[\int_0^{T}|b_1(s, X^{\alpha_2}(s))& - b_{1,n}(s, X^{\alpha_2}(s))|^4 \mathrm{d}s\Big]^{1/2}  \le C\E\Big[ \int_0^{T}|b_1(s, x_0+ \sigma B(s)) - b_{1,n}(s,x_0+ \sigma B(s))|^4\mathrm{d}s \Big]^{1/2}
	\end{align}
	and using the density of Brownian motion, we have for every $p\ge 1$
	\begin{align*}
		\E\Big[\Big| b_{1}(s,x_0+ \sigma B(s))- b_{1,n}(s,x_0+\sigma B(s))&\Big|^p\Big]= \frac{1}{\sqrt{2\pi s}}\int_{\mathbb{R}^d}\Big|b_{1,n} (s,x_0+\sigma y)-b_{1} (s,x_0+\sigma y)\Big|^pe^{-\frac{|y|^2}{2s}}\diffns y\\
		=&\frac{1}{\sqrt{2\pi s}}\int_{\mathbb{R}^d}\Big|b_{1,n} (s,\sigma y)-b_{1} (s,\sigma y)\Big|^pe^{-\frac{|y-x_0|^2}{2s}}\diffns y\\
		=&\frac{1}{\sqrt{2\pi s}}\int_{\mathbb{R}^d}\Big|b_{1,n} (s,\sigma y)-b_{1} (s,\sigma y)\Big|^pe^{-\frac{|y-2x_0|^2}{4s}}e^{-\frac{|y|^2}{4s}}e^{\frac{|x_0|^2}{2s}}\diffns y\\
		\leq &\frac{1}{\sqrt{2\pi s}}e^{\frac{|x_0|^2}{2s}}\int_{\mathbb{R}^d}\big|b_{1,n} (s,\sigma y)-b_{1} (s,\sigma y)\big|^pe^{-\frac{|y|^2}{4s}}\diffns y.
	\end{align*}
	By Fubini's theorem, this shows that
	\begin{multline}
	\label{eq:estim.bnb}
		\E\Big[\int_0^{T}|b_1(s, X^{\alpha_2}(s)) - b_{1,n}(s, X^{\alpha_2}(s))|^2\mathrm{d}s\Big]\\ \le C\Big(\int_0^T\frac{1}{\sqrt{2\pi s}}e^{\frac{|x_0|^2}{2s}}\int_{\mathbb{R}^d}\big|b_{1,n} (s,\sigma y)-b_{1} (s,\sigma y)\big|^4e^{-\frac{|y|^2}{4s}}\diffns y\diff s\Big)^{1/2}.
	\end{multline}
	Let us now turn our attention to the first term in \eqref{eq:estim.diff.x}.
	Since $\Lambda(\lambda,t)$ takes the form 
	\begin{align*}
		\Lambda(\lambda, t) &= x+ \int_0^t\Big\{\lambda b_{n}(s, X^{\alpha_1}_n(s),\alpha_1(s)) + (1-\lambda)b(s, X^{\alpha_2}(s),{\alpha_2}(s))\Big\}\diff s + \sigma B(t)\\
		& = x +\int_0^tb^{\lambda,\alpha_2}(s)\mathrm{d}s + \sigma B(t).
	\end{align*} 
	we use Jensen inequality, Girsanov's theorem as above and Lipschitz continuity of $b_2$ to get
	\begin{align}
		\notag
		\E\Big[&\exp\Big(2\int_{s}^t\int_0^1\partial_{x}b_{1,n}(r, \Lambda_n(\lambda,r))+\partial_xb_2(r,\Lambda_n(\lambda,r),\alpha_1(r))\diff\lambda\mathrm{d}r\Big) \Big]\\\notag
		&\le C\int_0^1 \E_{\mathbb{Q}^\lambda}\Big[\exp\Big(4\int_{s}^t\partial_xb_{1,n}(r, \Lambda_n(\lambda,r))\mathrm{d}r \Big) \Big]^{1/2}\diff\lambda\\
		\label{eq:estime.bprime}
		&\le C\int_0^1\E\Big[\exp\Big(4\int_{s}^t\partial_xb_{1,n}(r, x_0+ \sigma B(r))\mathrm{d}r \Big) \Big]^{1/2}\diff\lambda,
	\end{align}
	with $\diff \mathbb{Q}^\lambda = \mathcal{E}\big(\frac{\sigma^\top}{|\sigma|^2}\int_0^Tb^{\lambda,\alpha_2}(s)\mathrm{d}B(s) \big)\diff \PP$, and where we used the fact that $b^{\lambda,\alpha_2}$ is bounded.
	Since the sequence $(b_{1,n})_n$ is uniformly bounded, it follows from Lemma \ref{lemmaexpoloc} that 
	\begin{align}
	\label{eq:bound.bprime}
		\sup_nE\Big[\exp\Big(4\int_{s}^t\partial_xb_{1,n}(r, x_0+ \sigma \cdot B(r))\mathrm{d}r \Big) \Big]
		\leq C.
	\end{align}
	Therefore, putting together \eqref{eq:estim.diff.x}, \eqref{eq:estim.alpha12}, \eqref{eq:estim.bnb} and \eqref{eq:bound.bprime} concludes the proof.
		
	Since $b_k$ is Lipschitz continuous the convergence (ii) follows by classical arguments, the proof is omitted.
\end{proof}
	
\begin{lemm}
\label{lem:J.continuous}
	Let $\alpha\in \mathcal{A}$ and let $\alpha_n$ be a sequence of admissible controls such that $\delta(\alpha_n,\alpha)\to 0$.
	Then, it holds
	\begin{itemize}
		\item[(i)] $| J_k(\alpha_n) - J_k(\alpha) | \to 0$ as $n\to \infty$  for every $k \in \mathbb{N}$ fixed. In particular, the function $J_k:(\mathcal{A},\delta) \to \mathbb{R}$ is continuous.
		\item[(ii)] $|J_n(\alpha) - J(\alpha)| \le \varepsilon_n$  for some $C>0$ with $\varepsilon_n\downarrow 0$.
	\end{itemize}
	\end{lemm}
	\begin{proof}
		(i)	The continuity of $J_k$ easily follows by Lipschitz continuity of $f$ and $g$.
		In fact, we have
		\begin{align*}
		|J_k(\alpha_n) - J_k(\alpha)|& \le \E\Big[|g(X^{\alpha_n}_k(T)) - g(X^{\alpha}_k(T))| + \int_0^T|f(t, X^{\alpha_n}_k(t), \alpha_n(t)) - f(t, X^{\alpha}_k(t), \alpha(t)) |\diff t \Big]\\
		&\le C\E\Big[|X^{\alpha_n}_k(T) - X^{\alpha}_k(T)| + \int_0^T|X^{\alpha_n}_k(t) - X_k^{\alpha}(t)| + |\alpha_n(t) - \alpha(t)| \diff t\Big] \to 0,
		\end{align*}
		where the convergence follows by dominated convergence and Lemma \ref{lem:conv.Xnn}.
		
		(ii) is also a direct consequence of Lemma \ref{lem:conv.Xnn} since Fubini's theorem and Lipschitz continuity of $f$ and $g$ used as in part (i) above imply
		\begin{align*}
		| J_n(\alpha) - J(\alpha) | &\le C\sup_{t\in [0,T]}\E[|X^{\alpha}_n(t) - X^{\alpha}(t)|]
		\le \varepsilon_n,
		\end{align*}
		where the second inequality follows from Lemma \ref{lem:conv.Xnn}.
	\end{proof}
	The next lemma pertains to  the stability of the adjoint process with respect to the drift and the control process.
	This result is based on similar stability properties for stochastic flows.
	Given $x \in \mathbb{R}$ and the solution $X^{\alpha,x}$ of the SDE \eqref{eqSpro1} with initial condition $X^{\alpha,x}_t = x$, the first variation process of $X^{\alpha,x}$ is the derivative $\Phi^\alpha(t,s)$ of the function $x\mapsto X^{\alpha,x}(s)$.
	Existence and properties of this Sobolev differentiable flow have been extensively studied by Kunita \cite{Kun90} for equations with sufficiently smooth coefficients.
	In particular, when the drift $b$ is Lipschitz and continuously differentiable, the function $\Phi^\alpha(t,s)$ exists and, for almost every $\omega$, is the (classical) derivative of $x\mapsto X^{\alpha,x}(s)$.
	The case of measurable (deterministic) drifts is studied by Mohammed et. al. \cite{MNP2015} and extended to measurable and random drifts in \cite{MenTan19}.
	These works show that, when $b$ is measurable, then $X^{\alpha,\cdot}(s)\in L^2(\Omega, W^{1,p}(U))$ for every $s \in [t,T]$ and $p>1$, where $W^{1,p}(U)$ is the usual Sobolev space and $U$ an open and bounded subset of $\mathbb{R}$.
	That is, $\Phi^\alpha(t,s)$ exists and is the weak derivative of $X^{\alpha,\cdot}$.

The proof of the stability result will make use of an explicit representation of the process $\Phi^\alpha$ with respect to the time-space local time.	
Recall that for $a\in \mathbb{R}$ and $X=\{X(t),t\geq 0\}$ a continuous semimartingale, the local time $L^{X}(t,a)$ of $X$ at $a$ is defined by the Tanaka-Meyer formula as
$$
	|X(t)-a|=|X(0)-a|+\int_0^t\sgn(X(s)-a)\diffns X(s) +L^{X}(t,a) ,
$$
where $\sgn(x)=-1_{(-\infty,0]}(x)+1_{(0,+\infty)}(x)$. 
The local time-space integral plays a crucial role in the representations of the Sobolev derivative of the flows of the solution to the SDE \eqref{eqSpro1}. 
It is defined for functions in the space $(\mathcal{H}_x, \|\cdot\|^x)$ defined (see e.g. \cite{Ein2000}) as the space of Borel measurable functions $f:[0,T]\times \mathbb{R}\rightarrow \mathbb{R} $ with the norm
\begin{align*}
	\left\|f\right\|_x&:=2\Big(\int_0^1\int_{\mathbb{R}}f^2(s,z)\exp(-\frac{|z-x|^2}{2s})\frac{\diffns s \diff z}{\sqrt{2\pi s}}\Big)^{\frac{1}{2}}\notag
	+\int_0^1\int_{\mathbb{R}}|z-x| |f(s,x)|\exp(-\frac{|z-x|^2}{2s})\frac{\diffns s \diff z}{s\sqrt{2\pi s}}.
\end{align*}
Since $b_1$ is bounded, we obviously have $b_1 \in \mathcal{H}^x$ for every $x$.
Moreover, it follows from \cite{Eisen07} (see also \cite{BMBPD17}) that for every continuous semimartingale $X$ the local time-space integral of $f\in \mathcal{H}^x$ with respect to $L^{X}(t,z)$ is well defined and satisfies
\begin{align}\label{eqtransLT1}
	\int_0^t\int_{\mathbb{R}}f(s,z) L^{X}(\diffns s,\diffns z) = - \int_0^t\partial_xf(s,X(s))\diffns \langle X\rangle_ s,
\end{align}
for every continuous function (in space)  $f \in \mathcal{H}^x$ admitting a continuous derivative $\partial_xf(s,\cdot)$,  see \cite[Lemma 2.3]{Eisen07}.
This representation allows to derive the following:
\begin{lemm}
\label{lem:bound.int.local.time}
	For every $\alpha \in \mathcal{A}$ and $c\ge0$, it holds
	\begin{equation*}
		\E\Big[e^{c\int_t^{s}\int_{\mathbb{R}}b_1\left(u,z\right)L^{X^{\alpha,x}}(\mathrm{d}u,\mathrm{d}z)}\Big] <\infty.
	\end{equation*}
\end{lemm}
\begin{proof}
	First observe that for every $n \in \mathbb{N}$, it follows by Cauchy-Schwarz inequality that
	\begin{align*}
		&\E\Big[e^{c\int_t^{s}\int_{\mathbb{R}}b_{1,n}\left(u,z\right)L^{X^{\alpha,x}}(\mathrm{d}u,\mathrm{d}z)}\Big]\\ 
		& = \E\Big[\mathcal{E}\Big(\frac{\sigma^\top}{|\sigma|^2}\int_0^Tb(s, X^{\alpha}(s),{\alpha}(s))\mathrm{d}B(s) \Big)^{1/2}\mathcal{E}\Big(\int_0^T\frac{\sigma^\top}{|\sigma|^2}b(s, X^{\alpha}(s),{\alpha}(s))\mathrm{d}B(s) \Big)^{-1/2}\\
		&\quad \times e^{6\int_t^{s}\int_{\mathbb{R}}b_{1,n}\left(u,z\right)L^{X^{\alpha,x}}(\mathrm{d}u,\mathrm{d}z)} \Big]\\
		& \le C\E_{\mathbb{Q}}\Big[ e^{2c\int_t^{s}\int_{\mathbb{R}}b_{1,n}\left(u,z\right)L^{X^{\alpha,x}}(\mathrm{d}u,\mathrm{d}z)} \Big]^{1/2}
	\end{align*}
	where $\mathbb{Q}$ is the probability measure given as in \eqref{eq:def.probab.Q} with $\alpha_2$ therein replaced by $\alpha$.
	Hence, since $(X^{\alpha,x}-x_0)\sigma^\top/|\sigma|^2$ is a Brownian motion under $\mathbb{Q}$, it follows by \eqref{eqtransLT1} that
	\begin{align*}
		E\Big[e^{c\int_t^{s}\int_{\mathbb{R}}b_{1,n}\left(u,z\right)L^{X^{\alpha,x}}(\mathrm{d}u,\mathrm{d}z)}\Big]&\le C\E_{\mathbb{Q}}\Big[ e^{-2c\|\sigma\|^2\int_t^{s}\partial_xb_{1,n}\left(u,X^{\alpha,x}(u)\right)\mathrm{d}u} \Big]^{1/2}\\
		&= C\E\Big[ e^{-2c\|\sigma\|^2\int_t^{s}\partial_xb_{1,n}\left(u,x_0 + \sigma B(u)\right)\mathrm{d}u} \Big]^{1/2}\le \overline C
	\end{align*}
	for some constant $\overline C>0$ which does not depend on $n$, where this latter inequality follows by Lemma \ref{lemmaexpoloc}.
	Since $b_1$ is bounded and $b_{1,n}$ converges to $b_1$ pointwise, it follows by \cite[Theorem 2.2]{Eisen07} that $\int_{\mathbb{R}}b_{1,n}\left(u,z\right)L^{X^{\alpha,x}}(\mathrm{d}u,\mathrm{d}z) \to \int_{\mathbb{R}}b_1\left(u,z\right)L^{X^{\alpha,x}}(\mathrm{d}u,\mathrm{d}z) $ as $n$ goes to infinity.
	Thus, it follows by dominated convergence that
	\begin{equation*}
		E\Big[e^{c\int_t^{s}\int_{\mathbb{R}}b_{1}\left(u,z\right)L^{X^{\alpha,x}}(\mathrm{d}u,\mathrm{d}z)}\Big] = \lim_{\to \infty}E\Big[e^{c\int_t^{s}\int_{\mathbb{R}}b_{1,n}\left(u,z\right)L^{X^{\alpha,x}}(\mathrm{d}u,\mathrm{d}z)}\Big]<\overline C.
	\end{equation*}
\end{proof}
	We are now ready to prove stability of the follow and of the adjoint processes.
\begin{lemm}
\label{lem:conv.y.phi}
	Let $\alpha\in \mathcal{A}$ and $\alpha_n$ be a sequence of admissible controls such that $\delta(\alpha_n,\alpha)\to 0$.
	Then, the processes $X^{\alpha_n}_n$ and $X^{\alpha}$ admit Sobolev differentiable flows denoted $\Phi^{\alpha_n}_n$ and $\Phi^{\alpha}$, respectively and for every $0\le t\le s\le T$ it holds
	\begin{itemize}
		\item[(i)] $\E\big[|\Phi^{\alpha_n}_n(t,s) - \Phi^\alpha(t,s) |^2 \big] \to 0$ as $n\to \infty$,
		\item[(ii)] $\E\big[| Y^{\alpha_n}_n(t) - Y^\alpha(t)| \big] \to 0$ as $n\to \infty$,
	\end{itemize}
	where $Y^\alpha$ is the adjoint process defined as
	\begin{equation*}
		Y^{\alpha}(t) := \E\Big[\Phi^{\alpha}(t,T) \partial_xg( X^{\alpha}(T)) + \int_t^T\Phi^{\alpha}(t,s) \partial_xf(s, X^{\alpha}(s), \alpha(s))\mathrm{d}s\mid \mathcal{F}_t \Big],
	\end{equation*}	
	and $Y^{\alpha_n}_n$ is defined similarly, with $(X^{\alpha},\alpha, \Phi^\alpha)$ replaced by  $(X^{\alpha_n}_n,\alpha_n, \Phi^{\alpha_n}_n)$.
\end{lemm}
\begin{proof} 
	The existence of the process $\Phi^{\alpha_n}_n$ is standard, it follows for instance by \cite{kunita01}.
	The existence of the flow $\Phi^{\alpha}$ follows by \cite[Theorem 1.3]{MenTan19}.
	We start by proving the first convergence claim.
	As explained above, these processes admit explicit representations in terms of the space-time local time process.
	It fact, it follows from Theorem \ref{Thmexpliflowder} that $\Phi^\alpha$ admits the representation 
	\begin{equation*} 
		\Phi^\alpha(t,s)  = e^{\int_t^{s}\int_{\mathbb{R}}b_1\left(u,z\right)L^{X^{\alpha,x}}(\mathrm{d}u,\mathrm{d}z)}e^{\int_t^{s}\partial_xb_2\left(u,X^{\alpha,x}(u),\alpha(u)\right) \mathrm{d}u}
	\end{equation*}
	and $\Phi_n^{\alpha_n}$ admits the same representation with $(b_1, X^{\alpha,x},\alpha)$ replaced by $(b_{1,n}, X^{\alpha_n,x}, \alpha_n)$.
	Using these explicit representations and H\"older inequality we have
	\begin{align*}
		&\E\Big[\Big|\Phi^{\alpha}(t,s)- \Phi_n^{\alpha_n}(t,s)\Big|^2\Big]\notag\\
		\le&	2\E\Big[\Big|e^{\int_t^{s}\int_{\mathbb{R}}b_1\left(u,z\right)L^{X^{\alpha,x}}(\mathrm{d}u,\mathrm{d}z)}\Big\{e^{\int_t^{s}\partial_xb_2\left(u,X^{\alpha,x}(u),\alpha(u)\right) \mathrm{d}u}
		-e^{\int_t^{s}\partial_xb_{2}\left(u,X_n^{\alpha_n,x}(u),\alpha_n(u)\right) \mathrm{d}u}\Big\}\Big|^2\Big]\notag\\
		&	+2\E\Big[\Big|e^{\int_t^{s}\partial_xb_{2}\left(u,X_n^{\alpha_n,x}(u),\alpha_n(u)\right) \mathrm{d}u}\Big\{e^{\int_t^{s}\int_{\mathbb{R}}b_1\left(u,z\right)L^{X^{\alpha,x}}(\mathrm{d}u,\mathrm{d}z)}
		-e^{\int_t^{s}\int_{\mathbb{R}}b_{1,n}\left(u,z\right)L^{X_n^{\alpha_n,x}}(\mathrm{d}u,\mathrm{d}z)}\Big\}\Big|^2\Big]\notag\\
		\leq &2	\E\Big[e^{4\int_t^{s}\int_{\mathbb{R}}b_1\left(u,z\right)L^{X^{\alpha,x}}(\mathrm{d}u,\mathrm{d}z)}\Big]^{\frac{1}{2}}\E\Big[\Big\{e^{\int_t^{s}\partial_xb_2\left(u,X^{\alpha,x}(u),\alpha(u)\right) \mathrm{d}u}
		-e^{\int_t^{s}\partial_xb_{2}\left(u,X_n^{\alpha_n,x}(u),\alpha_n(u)\right) \mathrm{d}u}\Big\}^4\Big]^{\frac{1}{2}}\notag\\
		&	+2\E\Big[e^{4\int_t^{s}\partial_xb_{2}\left(u,X_n^{\alpha_n,x}(u),\alpha_n(u)\right) \mathrm{d}u}\Big]^{\frac{1}{2}}\E\Big[\Big\{e^{\int_t^{s}\int_{\mathbb{R}}b_1\left(u,z\right)L^{X^{\alpha,x}}(\mathrm{d}u,\mathrm{d}z)}
		-e^{\int_t^{s}\int_{\mathbb{R}}b_{1,n}\left(u,z\right)L^{X_n^{\alpha_n,x}}(\mathrm{d}u,\mathrm{d}z)}\Big\}^4\Big]^{\frac{1}{2}}.
	\end{align*}
	Splitting up the terms in power 4, then applying H\"older and Young's inequality we continue the estimations as
	\begin{align}
		&\E\Big[\Big|\Phi^{\alpha}(t,s)- \Phi_n^{\alpha_n}(t,s)\Big|^2\Big]\notag\\
		\leq & 2^7\E\Big[e^{4\int_t^{s}\int_{\mathbb{R}}b_1\left(u,z\right)L^{X^{\alpha,x}}(\mathrm{d}u,\mathrm{d}z)}\Big]^{\frac{1}{2}}  \E\Big[\Big\{e^{6\int_t^{s}\partial_xb_2\left(u,X^{\alpha,x}(u),\alpha(u)\right) \mathrm{d}u}
		+e^{6\int_t^{s}\partial_xb_{2}\left(u,X_n^{\alpha_n,x}(u),\alpha_n(u)\right) \mathrm{d}u}\Big\}\Big]^{\frac{1}{4}}
		\notag	\\
		&\times \E\Big[\Big\{e^{\int_t^{s}\partial_xb_2\left(u,X^{\alpha,x}(u),\alpha(u)\right) \mathrm{d}u}%
		-e^{\int_t^{s}\partial_xb_{2}\left(u,X_n^{\alpha_n,x}(u),\alpha_n(u)\right) \mathrm{d}u}\Big\}^2\Big]^{\frac{1}{4}}\notag\\
		&	+2^7 \E\Big[\Big|e^{4\int_t^{s}\partial_xb_{2}\left(u,X_n^{\alpha_n,x}(u),\alpha_n(u)\right) \mathrm{d}u}\Big]^{\frac{1}{2}}\E\Big[\Big\{e^{6\int_t^{s}\int_{\mathbb{R}}b_1\left(u,z\right)L^{X^{\alpha,x}}(\mathrm{d}u,\mathrm{d}z)}
		+e^{6\int_t^{s}\int_{\mathbb{R}}b_{1,n}\left(u,z\right)L^{X_n^{\alpha_n,x}}(\mathrm{d}u,\mathrm{d}z)}\Big\}\Big]^{\frac{1}{4}}
		\notag\\
		&\times \E\Big[\Big\{e^{\int_t^{s}\int_{\mathbb{R}}b_1\left(u,z\right)L^{X^{\alpha,x}}(\mathrm{d}u,\mathrm{d}z)}
		-e^{\int_t^{s}\int_{\mathbb{R}}b_{1,n}\left(u,z\right)L^{X_n^{\alpha_n,x}}(\mathrm{d}u,\mathrm{d}z)}\Big\}^2\Big]^{\frac{1}{4}}\notag\\
		=&CI_1^{\frac{1}{2}}\times I^{\frac{1}{2}}_{2,n}\times I^{\frac{1}{4}}_{3,n} +CI^{\frac{1}{2}}_{4,n}\times I^{\frac{1}{4}}_{5,n}\times I^{\frac{1}{4}}_{6,n}.
		\end{align}
	It follows from Lemma \ref{lem:bound.int.local.time} that  $I_1$ and $I_{5,n}$ are bounded.
	Since $\partial_xb_2$ is bounded, it follows that $I_{2,n}$ and $I_{4,n}$ are also bounded with bounds independent on $n$.  Let us now show that $I_{3,n}$ and $I_{6,n} $ converge to zero. 
	We show  only the convergence of $I_{6,n}$ since that of $I_{3,n}$ will follow (at least for a subsequence) from Lemma \ref{lem:conv.Xnn} and dominated convergence since $\partial_xb_{2}$ is continuous and bounded. 
		
	To that end, further define the processes $A_n^{\alpha_n}$ and $A^{\alpha}$ by 
	$$
		A_n^{\alpha_n}(t,s):=e^{\int_t^{s}\int_{\mathbb{R}}b_{1,n}\left(u,z\right)L^{X_n^{\alpha_n,x}}(\mathrm{d}u,\mathrm{d}z)}\quad \text{and} \quad  A^{\alpha}(t,s):=e^{\int_t^{s}\int_{\mathbb{R}}b_1\left(u,z\right)L^{X^{\alpha,x}}(\mathrm{d}u,\mathrm{d}z)}.
	$$ 
	In order to show that $A_n^{\alpha_n}$ converges to $A^{\alpha}$ in $L^2$, we will show that $A_n^{\alpha_n}$ converges weakly to $A^{\alpha}$ in $L^2$ and that $E[|A_n^{\alpha_n}|^2]$ converges to $E[|A^{\alpha}|^2]$ in $\mathbb{R}$. 
	We first prove the weak convergence. 
	Since the set
	$$
		\Big\{\mathcal{E}\Big(\int_0^1\dot{\varphi}(s)\mathrm{d}B(s)\Big):\varphi\in C^{1}_b([0,T],\mathbb{R}^d)\Big\}
	$$
	spans a dense subspace in $L^2(\Omega)$, in order to show weak convergence, it is enough to show that 
	$$
		E\Big[A_n^{\alpha_n}(t,s) \mathcal{E}\Big(\int_0^1\dot{\varphi}(s)\mathrm{d}B(s)\Big)\Big]\rightarrow E\Big[A^{\alpha}(t,s) \mathcal{E}\Big(\int_0^1\dot{\varphi}(s)\mathrm{d}B(s)\Big)\Big]\quad \text{for every}\quad \varphi\in C^{1}_b([0,T],\mathbb{R}^d).
	$$
	Denote by $\tilde X_n^{\alpha_n, x}$ and $\tilde X^{\alpha,x}$ the processes given by
	\begin{align}\label{eqxntilde10}
		\diffns \tilde X^{\tilde \alpha_n,x}_n(t)= \Big(b_{1,n}(t,\tilde X^{\tilde \alpha_n,x}_n(t))+ b_{2}(t,\tilde X^{\tilde \alpha_n,x}_n(t),\tilde \alpha_n)+\sigma \dot\varphi(t) \Big)\diffns t
		+\sigma \diffns B(t),
	\end{align} 
	and
	\begin{align}\label{eqxntilde11}
		\diffns \tilde X^{\tilde \alpha,x}(t) = \Big(b_{1}(t,\tilde X^{\tilde \alpha,x}(t))+ b_{2}(t,\tilde X^{\tilde \alpha,x}(t),\tilde \alpha_n)+\sigma \dot\varphi(t)\Big)\diffns t
		+\sigma \diffns B(t).
	\end{align} 
	Observe that these processes are well-defined, since we have $\tilde X^{\tilde \alpha,x}(t,\omega) = X^{\alpha,x}(t,\omega + \varphi)$ and $\tilde X_n^{\tilde \alpha_n,x}(t,\omega) = X_n^{\alpha_n,x}(t,\omega + \varphi)$.
	Using the Cameron-Martin-Girsanov theorem as in the proof of Lemma \ref{lem:bound.int.local.time}, we have 
	\begin{align*}
		&	\Big|E\Big[\mathcal{E}\Big(\int_0^T\dot{\varphi}(s)\mathrm{d}B(s)\Big)\Big\{A_n^{\alpha_n}(t,s)-A^{\alpha}(t,s)\Big\}\Big]\Big|\notag\\
		=&\Big|E\Big[e^{\int_s^{t}\int_{\mathbb{R}}b_{1,n}\left(u,z\right)L^{\tilde X_n^{\tilde\alpha_n,x}}(\mathrm{d}u,\mathrm{d}z)}-e^{\int_s^{t}\int_{\mathbb{R}}b_1\left(u,z\right)L^{\tilde X^{\tilde\alpha,x}}(\mathrm{d}u,\mathrm{d}z)}\Big]\Big|\notag\\
		=&\Big|\E\Big[\mathcal{E}\Big(\int_0^T\Big\{\tilde u_n(s,x+\sigma \cdot B(s),\alpha_n (s))+\sigma \cdot \dot{\varphi}(s)\Big\}\mathrm{d}B(s)\Big)e^{\int_s^{t}\int_{\mathbb{R}}b_{1,n}\left(u,z\right)L^{|\sigma\|B^x_\sigma}(\mathrm{d}u,\mathrm{d}z)}\notag\\
		&\quad -\mathcal{E}\Big(\int_0^T\Big\{\tilde u(s,x+\sigma \cdot B(s),\alpha (s))+\sigma \cdot \dot{\varphi}(s)\Big\}\mathrm{d}B(s)\Big)e^{\int_s^{t}\int_{\mathbb{R}}b_1\left(u,z\right)L^{|\sigma\|B^x_\sigma}(\mathrm{d}u,\mathrm{d}z)}\Big]\Big|	,
	\end{align*}
	where $\tilde u(s, x,\alpha(\omega)): = u(s, x,\alpha(\omega+\varphi))$.
	Next, add and subtract the same term and then use the inequality $|e^x-e^y|\leq |x-y||e^x+e^y|$ and then H\"older inequality and putting
	\begin{equation*}\label{eqnewbm1}
		u(s, x,\alpha(\omega)): = (\frac{\sigma^1b}{|\sigma|^2}, \dots, \frac{\sigma^db}{|\sigma|^2})(t,x,\alpha(\omega))\quad \text{and}\quad B^x_\sigma:=x+\sum_{i=1}^d\frac{\sigma_i}{\|\sigma\|}B^i,
	\end{equation*}
	we obtain
	\begin{align*}
		&\Big|\E\Big[\mathcal{E}\Big(\int_0^T\dot{\varphi}(s)\mathrm{d}B(s)\Big)\Big\{A_n^{\alpha_n}(t,s)-A^{\alpha}(t,s)\Big\}\Big]\Big|\notag\\
		\leq &	\Big|E\Big[\mathcal{E}\Big(\int_0^T\{u_n(s,x+\sigma \cdot B(s),\alpha (s,\omega+\varphi))+\sigma \cdot \dot{\varphi}(s)\}\mathrm{d}B(s)\Big)\notag\\
		&\Big|\int_s^{t}\int_{\mathbb{R}}b_{1,n}\left(u,z\right)L^{\|\sigma\|B^x_\sigma}(\mathrm{d}u,\mathrm{d}z)-\int_s^{t}\int_{\mathbb{R}}b_1\left(u,z\right)L^{\|\sigma\|B^x_\sigma}(\mathrm{d}u,\mathrm{d}z)\Big|\\
		&\times\Big(e^{\int_s^{t}\int_{\mathbb{R}}b_{1,n}\left(u,z\right)L^{\|\sigma\|B^x_\sigma}(\mathrm{d}u,\mathrm{d}z)}+e^{\int_s^{t}\int_{\mathbb{R}}b_1\left(u,z\right)L^{\|\sigma\|B^x_\sigma}(\mathrm{d}u,\mathrm{d}z)}\Big)\Big]\Big|\notag\\
		&+\Big|E\Big[e^{\int_s^{t}\int_{\mathbb{R}}b_1\left(u,z\right)L^{\|\sigma\|B^x_\sigma}(\mathrm{d}u,\mathrm{d}z)}\Big\{\mathcal{E}\Big(\int_0^T\{u_n(s,x+\sigma \cdot B(s),\alpha_n (s,\omega+\varphi))+\sigma \cdot \dot{\varphi}(s)\}\mathrm{d}B(s)\Big)\notag\\
		&-\mathcal{E}\Big(\int_0^T\{u(s,x+\sigma \cdot B(s),\alpha (s,\omega+\varphi))+\sigma \cdot \dot{\varphi}(s)\}\mathrm{d}B(s)\Big)\Big\}\Big]\Big|.
	\end{align*}
	Therefore, another application of H\"older's inequality yields the estimate
	\begin{align}\label{eq:estimates.J}
		\notag
		&	\Big|\E\Big[\mathcal{E}\Big(\int_0^T\dot{\varphi}(s)\mathrm{d}B(s)\Big)\Big\{A_n^{\alpha_n}(t,s)-A^{\alpha}(t,s)\Big\}\Big]\Big|\\\notag
		\leq &	4\E\Big[\mathcal{E}\Big(\int_0^T\{u_n(s,x+\sigma \cdot B(s),\alpha_n (s,\omega+\varphi))+\sigma \cdot \dot{\varphi}(s)\}\mathrm{d}B(s)\Big)^4\Big]^{\frac{1}{4}}\notag\\
		&\E\Big[\Big|\int_s^{t}\int_{\mathbb{R}}\Big(b_{1,n}\left(u,z\right)-b_1\left(u,z\right)\Big)L^{\|\sigma\|B^x_\sigma}(\mathrm{d}u,\mathrm{d}z)\Big|^2\Big]^{\frac{1}{2}}\notag\\
		&\times \E\Big[e^{4\int_s^{t}\int_{\mathbb{R}}b_{1,n}\left(u,z\right)L^{\|\sigma\|B^x_\sigma}(\mathrm{d}u,\mathrm{d}z)}+e^{4\int_s^{t}\int_{\mathbb{R}}b_1\left(u,z\right)L^{\|\sigma\|B^x_\sigma}(\mathrm{d}u,\mathrm{d}z)}\Big]^{\frac{1}{4}}\notag\\
		&+\E\Big[e^{2\int_s^{t}\int_{\mathbb{R}}b_1\left(u,z\right)L^{\|\sigma\|B^x_\sigma}(\mathrm{d}u,\mathrm{d}z)}\Big]^{\frac{1}{2}}\E\Big[\Big\{\mathcal{E}\Big(\int_0^T\{u_n(s,x+\sigma \cdot B(s),\alpha_n (s,\omega+\varphi))+\dot{\varphi}(s)\}\mathrm{d}B(s)\Big)\notag\\
		&-\mathcal{E}\Big(\int_0^T\{u(s,x+\sigma \cdot B(s),\alpha (s,\omega+\varphi))+\sigma \cdot \dot{\varphi}(s)\}\mathrm{d}B(s)\Big)\Big\}^2\Big]^{\frac{1}{2}}\notag\\
		=&J_{1,n}^{\frac{1}{4}}\times J_{2,n}^{\frac{1}{2}}\times J_{3,n}^{\frac{1}{4}}+J_{4,n}^{\frac{1}{2}}\times J_{5,n}^{\frac{1}{2}}.
	\end{align}
	Using Lemma \ref{Lemmbound1}, it follows that $J_{2,n}$ converge to zero, and by dominated convergence $J_{5,n}$ also convergences to zero. 
	Thanks to Lemma \ref{lemmaexpoloc} and boundedness of $b_{1,n}$ (respectively $b_1$), the term $J_{3,n}$ (respectively $J_{4,n}$) is bounded. 
	The bound of $J_{1,n}$ follows by the uniform boundedness of $u_n$. 
		
	It remains to show that $\E[|A_n^{\alpha_n}(t)|^2]$ converges to $\E[|A^{\alpha}(t)|^2]$ in $\mathbb{R}$. 	Using Girsanov transform as in the proof of Lemma \ref{lem:bound.int.local.time}, we have
	\begin{align}
		\E[|A_n^{\alpha_n}(t)|^2]=&\E\Big[e^{2\int_s^{t}\int_{\mathbb{R}}b_{1,n}\left(u,z\right)L^{X_n^{\alpha,x}}(\mathrm{d}u,\mathrm{d}z)}\Big]\notag\\
		=& \E\Big[\mathcal{E}\Big(\int_0^T\{u_n(s,x+\sigma\cdot B(s),\alpha_n (s,\omega+\varphi))+\sigma\cdot\dot{\varphi}(s)\}\mathrm{d}B(s)\Big)e^{2\int_s^{t}\int_{\mathbb{R}}b_{1,n}\left(u,z\right)L^{\|\sigma\|B^x_\sigma}(\mathrm{d}u,\mathrm{d}z)}\Big]
	\end{align}
	and 
	\begin{align}
		\E[|A^{\alpha}(t)|^2]=&\E\Big[e^{2\int_s^{t}\int_{\mathbb{R}}b_{1}\left(u,z\right)L^{X^{\alpha,x}}(\mathrm{d}u,\mathrm{d}z)}\Big]\notag\\
		=& \E\Big[\mathcal{E}\Big(\int_0^T\{u(s,x+\sigma\cdot  B(s),\alpha (s,\omega+\varphi))+\sigma\cdot\dot{\varphi}(s)\}\mathrm{d}B(s)\Big)e^{2\int_s^{t}\int_{\mathbb{R}}b_{1}\left(u,z\right)L^{\|\sigma\|B^x_\sigma}(\mathrm{d}u,\mathrm{d}z)}\Big].
	\end{align}
	Therefore using once more $|e^x-e^y|\leq |x-y||e^x+e^y|$ and Cauchy-Schwarz inequality 
	\begin{align}
		&|\E[|A_n^{\alpha_n}(t)|^2]-\E[|A^{\alpha}(t)|^2]|\notag\\
		=&\Big|\E\Big[\mathcal{E}\Big(\int_0^T\{u_n(s,x+\sigma\cdot B(s),\alpha_n (s,\omega+\varphi))+\sigma\cdot\dot{\varphi}(s)\}\mathrm{d}B(s)\Big)e^{2\int_s^{t}\int_{\mathbb{R}}b_{1,n}\left(u,z\right)L^{\|\sigma\|B^{x}_\sigma}(\mathrm{d}u,\mathrm{d}z)}\Big]\notag\\
		&-\E\Big[\mathcal{E}\Big(\int_0^T\{u(s,x+\sigma\cdot B(s),\alpha (s,\omega+\varphi))+\sigma\cdot\dot{\varphi}(s)\}\mathrm{d}B(s)\Big)e^{2\int_s^{t}\int_{\mathbb{R}}b_{1}\left(u,z\right)L^{\|\sigma\|B^x_\sigma}(\mathrm{d}u,\mathrm{d}z)}\Big]\Big|\notag\\
		\leq& \Big|\E\Big[e^{4\int_s^{t}\int_{\mathbb{R}}b_{1,n}\left(u,z\right)L^{\|\sigma\|B^{x}_\sigma}(\mathrm{d}u,\mathrm{d}z)}\Big]^{\frac{1}{2}}\E\Big[
		\mathcal{E}\Big(\int_0^T\{u_n(s,x+\sigma\cdot B(s),\alpha_n (s,\omega+\varphi))+\sigma\cdot\dot{\varphi}(s)\}\mathrm{d}B(s)\Big)\notag\\
		&-\mathcal{E}\Big(\int_0^T\{u(s,x+\sigma\cdot B(s),\alpha (s,\omega+\varphi))+\sigma\cdot\dot{\varphi}(s)\}\mathrm{d}B(s)\Big)^2\Big]^{\frac{1}{2}}\Big|\notag\\
		&+C\Big|\E\Big[\Big(\int_s^{t}\int_{\mathbb{R}}\{b_{1,n}\left(u,z\right)-b_{1}\left(u,z\right)\}L^{\|\sigma\|B^{x}}(\mathrm{d}u,\mathrm{d}z)\Big)^2\Big]^{\frac{1}{2}}\notag\\
		&\times \Big(\E\Big[e^{8\int_s^{t}\int_{\mathbb{R}}b_{1,n}\left(u,z\right)L^{\|\sigma\|B^{x}_\sigma}(\mathrm{d}u,\mathrm{d}z)}\Big]^{\frac{1}{4}}+\E\Big[e^{8\int_s^{t}\int_{\mathbb{R}}b_{1}\left(u,z\right)L^{\|\sigma\|B^{x}_\sigma}(\mathrm{d}u,\mathrm{d}z)}\Big]^{\frac{1}{4}}\Big)\notag\\
		&\times \E\Big[\mathcal{E}\Big(\int_0^T\{u(s,x+\sigma\cdot B(s),\alpha (s,\omega+\varphi))+\sigma\cdot\dot{\varphi}(s)\}\mathrm{d}B(s)\Big)^4\Big]^{\frac{1}{4}}\Big|\notag.
	\end{align}
	Now, introducing the random variables 
	\begin{align*}
		V_n := &\int_0^T\Big(u_n(s,x+\sigma \cdot B(s),\alpha_n (s,\omega+\varphi))-u(s,x+\sigma \cdot B(s),\alpha (s,\omega+\varphi))\Big)\mathrm{d}B(s)\\
		&-\frac{1}{2}\int_0^T\Big(|u_n(s,x+\sigma \cdot B(s),\alpha_n (s,\omega+\varphi))+\sigma\cdot\dot{\varphi}(s)|^2\notag\\
		&-|u(s,x+\sigma \cdot B(s),\alpha (s,\omega+\varphi))+\sigma\cdot\dot{\varphi}(s)|^2\Big)\mathrm{d} s
	\end{align*}
	and 
	\begin{align*}
		F_{1,n} := \int_s^{t}\int_{\mathbb{R}}\{b_{1,n}\left(u,z\right)-b_{1}\left(u,z\right)\}L^{\|\sigma\|B^{x}}(\mathrm{d}u,\mathrm{d}z)
	\end{align*}
	we continue the above estimations as
	\begin{align}
		|\E[|A_n^{\alpha_n}(t)|^2] &- \E[|A^{\alpha}(t)|^2]|\notag	\\
		&\leq CE\Big[V_n^2\Big\{\mathcal{E}\Big(\int_0^T\{u_n(s,x+\sigma \cdot B(s),\alpha_n (s,\omega+\varphi))+\sigma\cdot\dot{\varphi}(s)\}\mathrm{d}B(s)\Big)\notag\\
		&\quad +	\mathcal{E}\Big(\int_0^T\{u(s,x+\sigma \cdot B(s),\alpha (s,\omega+\varphi))+\sigma\cdot\dot{\varphi}(s)\}\mathrm{d}B(s)\Big)\Big\}^2\Big]\notag\\
		&\quad+C\Big|\E\Big[|F_{1,n}|^2\Big]^{\frac{1}{2}} \Big(E\Big[e^{8\int_s^{t}\int_{\mathbb{R}}b_{1,n}\left(u,z\right)L^{\|\sigma\|B^{x}_\sigma}(\mathrm{d}u,\mathrm{d}z)}\Big]^{\frac{1}{4}}+	\E\Big[e^{8\int_s^{t}\int_{\mathbb{R}}b_{1}\left(u,z\right)L^{\|\sigma\|B^{x}_\sigma}(\mathrm{d}u,\mathrm{d}z)}\Big]^{\frac{1}{4}}\Big)\notag\\
		&\quad \times \E\Big[\mathcal{E}\Big(\int_0^T\{u(s,x+\sigma\cdot B(s),\alpha (s,\omega+\varphi))+\sigma\cdot\dot{\varphi}(s)\}\mathrm{d}B(s)\Big)^4\Big]^{\frac{1}{4}}\Big|.
	\end{align}		
	By Lemma \ref{Lemmbound1}, $F_{1,n}$ converges to zero in $L^2(\Omega)$.
	Using similar arguments as in \cite[Lemma A.3]{BMBPD17}, one can show that $V_n$ converges to zero in $L^2(\Omega)$ by the boundedness of $u_n$ and the definition of the distance $\delta$. Observe however that in this case, $u_n$ depends on $\alpha_n$ and not on $\alpha$ as in \cite[Lemma A.3]{BMBPD17}. Nevertheless using the fact that $b_{1,n}$, $b_1$ and $b_2$ are bounded and Lipschitz in the second variable, one can show by dominated convergence theorem and similar reasoning as in \eqref{eq:estim.alpha12} that the overall term converges to zero. It is also worth mentioning that the other terms are uniformly bounded by application of either Girsanov theorem and/or Lemma \ref{lemmaexpoloc} to the uniformly bounded senquences $(u_n)_{n\geq 1},(b_{1,n})_{n\geq 1}$ and the bounded functions $u, b_{1}$.
		
	\vspace{.2cm}
	Let us now turn our attention to the proof of (ii).
	Compute the difference $Y_{n}^{\alpha_n}(t)-Y^\alpha(t)$, add and subtract the terms $\Phi^{\alpha}(t,T) \partial_xg(X_n^{\alpha_n}(T))$ and $\int_t^T\Phi^{\alpha}(t,u) \partial_xf(u,X_n^{\alpha_n}(u), \alpha_n(u))\diff u$ and then apply H\"older's inequality to obtain
	\begin{align}\label{eq:conv.Y}
		\notag
		&\E[|Y_{n}^{\alpha_n}(t)-Y^\alpha(t)|]\\\notag
		\leq & C_T\Big\{\E\Big[\Big|\Phi^{\alpha}(t,T)\Big|^2\Big]^{\frac{1}{2}}\E\Big[|\partial_xg( X_n^{\alpha_n}(T)) - \partial_xg( X^{\alpha}(T))|^2\Big]^{\frac{1}{2}}\\ 
		&+\E\Big[|\partial_xg( X_n^{\alpha_n}(T))|^2\Big]^{\frac{1}{2}}\E\Big[\Big|\Phi_n^{\alpha_n}(t,T) - \Phi^{\alpha}(t,T)\Big|^2\Big]^{\frac{1}{2}}\notag\\\notag
		&+\E\Big[\int_t^T|\Phi^{\alpha}(t,u)|^2\diff u\Big]^{\frac{1}{2}} \E\Big[\int_0^T|\partial_xf(u, X^{\alpha}(u), \alpha(u))-\partial_xf(u, X_n^{\alpha_n}(u), \alpha_n(u))|^2\diff u\Big]^{\frac{1}{2}}\notag\\
		&+\E\Big[\int_0^T|\partial_xf(u, X_n^{\alpha_n}(u), \alpha_n(u))|^2\diff u\Big]^{\frac{1}{2}}\E\Big[\int_0^T|\Phi_n^{\alpha_n}(u)-\Phi^{\alpha}(u)|^2\diff u\Big]^{\frac{1}{2}}\Big\}%
	\end{align}
	for some constant $C_T$ depending only on $T$.
	Since the process $\Phi^{\alpha}$ is square integrable, (see \cite[Theorem 1.3]{MenTan19}) it follows by boundedness and continuity of $\partial_xg,\partial_xf$ as well as Lemma \ref{lem:conv.Xnn} that the first and third terms converge to zero as $n$ goes to infinity.
	Moreover, by boundedness of $\partial_xf$ and $\partial_xg$ and the $L^2$ convergence of $\Phi_n^{\alpha_n}(t,u)$ to $\Phi^{\alpha}(t,u)$ given in part (i), we conclude that the second and last terms in \eqref{eq:conv.Y} converge to zero, which shows (ii).
\end{proof}
\begin{proof}(of Theorem \ref{thm:necc})
	Let $\hat\alpha$ be an optimal control and $n\ge 1$ fixed.
	Observe that by the linear growth assumption on $f,g$ the function $J_n$ is bounded from above.
	By Lemma \ref{lem:J.continuous} the function $J_n$ is also continuous on $(\mathcal{A},\delta)$ and there exists $\varepsilon_n$ such that 
	\begin{equation*}
		J(\hat\alpha) - J_n(\hat\alpha)\le \varepsilon_n \text{ and } J_n(\alpha) - J(\alpha) \le \varepsilon_n\quad \text{for all } \alpha \in \mathcal{A}.
	\end{equation*}
	That is, $J_n(\hat\alpha) \le \inf_{\alpha \in \mathcal{A}}J_n(\alpha) + 2\varepsilon_n$.
	Thus, by Ekeland's variational principle, see e.g. \cite{Ekeland79}, there is a control $\hat\alpha_n \in \mathcal{A}$ such that $\delta(\hat\alpha, \hat\alpha_n)\le (2\varepsilon_n)^{1/2}$ and 
	\begin{equation*}
		J_n(\hat\alpha_n) \le J_n(\alpha) + (2\varepsilon_n)^{1/2}\delta(\hat\alpha_n,\alpha)\quad \text{for all}\quad \alpha \in \mathcal{A}.
	\end{equation*}
	In other words, putting $J^\varepsilon_n(\alpha):= J_n(\alpha) + (2\varepsilon_n)^{1/2}\delta(\hat\alpha_n,\alpha)$, the control process $\hat\alpha_n$ is optimal for the problem with cost function $J^\varepsilon_n$.
		
	Now, let $\beta \in \mathcal{A}$ be an arbitrary control and $\varepsilon>0$ a fixed constant.
	By convexity of $\mathbb{A}$, it follows that  $\hat\alpha_n + \varepsilon\eta \in \mathcal{A}$, with $\eta := \beta - \hat\alpha_n$.
	Thus, since $b_n$ is sufficiently smooth, it is standard that the functional $J_n$ is G\^ateau differentiable (see
	\cite[Lemma 4.8]{MR3629171}) and its G\^ateau derivative in the direction $\eta$ is given by
	\begin{align*}
		\frac{d}{d\varepsilon}J_n(\alpha + \varepsilon \eta)_{|_{\varepsilon = 0}}& = \E\Big[\int_0^T\partial_xf(t, X_n^{\hat\alpha_n}(t), \hat\alpha_n(t))V_n(t) + \partial_{\alpha}f(t, X_n^{\hat\alpha_n}(t), \hat\alpha_n(t))\eta(t)\mathrm{d}t\\
		&\qquad + \partial_xg(X_n^{\hat\alpha_n}(T))V_n(T) \Big] ,
	\end{align*}
	where $V_n$ is the stochastic process solving the linear equation
	\begin{equation*}
		dV_n(t) = \partial_xb_n(t, X_n^\alpha(t),\alpha(t))V_n(t)\mathrm{d}t + \partial_\alpha b_n(t, X_n^\alpha(t),\alpha(t))\eta(t)\mathrm{d}t,\quad V_n(0) = 0.
	\end{equation*}
	On the other hand, we have
	\begin{equation*}
		\lim_{\varepsilon\downarrow 0}\frac{1}{\varepsilon}\big(\delta(\hat\alpha_n, \alpha + \varepsilon\eta) - \delta(\hat\alpha_n, \alpha) \big) \le C_M\E\big[ \sup_{t \in [0,T]}|\eta(t)|^2 \big]^{1/2}.
	\end{equation*}
	for a constant $C_M>0$ depending on the constant $M$ (introduced in the definition of $\mathcal{A}$).
	Therefore, $J^\varepsilon_n$ is also G\^ateau differentiable and since $\hat\alpha_n$ is optimal for $J^\varepsilon_n$, we have
	\begin{align*}
		0\le \frac{\mathrm{d}}{\mathrm{d}\varepsilon}J^\varepsilon_n(\hat\alpha_n + \varepsilon \eta)_{|_{\varepsilon = 0}} &= \frac{\mathrm{d}}{\mathrm{d}\varepsilon}J_n(\hat\alpha_n + \varepsilon \eta)_{|_{\varepsilon = 0}} + \lim_{\varepsilon\downarrow 0} (2\varepsilon_n)^{1/2}\frac{1}{\varepsilon}\delta(\hat\alpha_n,\hat\alpha_n + \varepsilon\eta)  \\
		& = \E\Big[\int_0^T\partial_xf\big(t, X_n^{\hat\alpha_n}(t), \hat\alpha_n(t) \big)V_n(t) + \partial_{\alpha}f\big( t, X_n^{\hat\alpha_n}(t), \hat\alpha_n(t) \big)\eta(t)\mathrm{d}t\\
		&\qquad + \partial_xg(X_n^{\hat\alpha_n}(T))V_n(T) \Big] + C_M\big(2\varepsilon_nE[\sup_t|\eta(t)|^2] \big)^{1/2}\\
		&\le \E\Big[\int_0^T\partial_\alpha H_n\big(t, X_n^{\hat\alpha}, Y_n^{\hat\alpha_n}(t), \hat\alpha_n(t) \big)\eta(t)\mathrm{d}t \Big] + C_M\varepsilon_n^{1/2},
	\end{align*}
	for some constant $M>0$.
	The inequality following since $\hat\alpha_n\in \mathcal{A}$, and where $H_n$ is the Hamiltonian of the problem with drift $b_n$ given by
	\begin{equation*}
		H_n(t,x,y,a) := f(t, x,a) + b_n(t,x,a)y
	\end{equation*}
	and $(Y^{\hat\alpha_n}_n, Z^{\hat\alpha_n}_n)$ the adjoint processes given by
	\begin{equation*}
		\mathrm{d}Y^{\hat\alpha_n}_n(t) = -\partial_xH_n(t, X_n^{\hat\alpha}, Y^{\hat\alpha_n}_n(t), \hat\alpha_n(t))\mathrm{d}t + Z^{\hat\alpha_n}_n(t)\mathrm{d}B(t).
	\end{equation*}
	By standard arguments, we can thus conclude that
	\begin{equation*}
		C_M\varepsilon_n^{1/2} +\partial_\alpha H_n(t, X_n^{\hat\alpha_n}(t), Y^{\hat\alpha_n}_n(t), \hat\alpha_n(t))\cdot (\beta - \hat\alpha_n(t)) \ge 0 \quad \PP\otimes \mathrm{d}t \mathrm{-a.s}.
	\end{equation*}
	Recalling that $b_{1,n}$ does not depend on $\alpha$, this amounts to
	\begin{equation*}
		C_M\varepsilon_n^{1/2} + \Big\{ \partial_{\alpha}f(t, X_n^{\hat\alpha_n}(t), \hat\alpha_n(t)) + \partial_{\alpha}b_2\big(t, X_n^{\hat\alpha_n}(t),  \hat\alpha_n(t) \big)Y^{\hat\alpha_n}_n(t) \Big\}\cdot(\beta - \hat\alpha_n(t)) \ge 0 \quad \PP \otimes dt\text{-a.s.}
	\end{equation*} 
	We will now take the limit on both sides above as $n$ goes to infinity.
	It follows by Lemma \ref{lem:conv.Xnn} and Lemma \ref{lem:conv.y.phi} respectively that $X_n^{\hat\alpha_n}(t) \to X^{\hat\alpha}(t)$ and $Y^{\hat\alpha_n}_n(t) \to Y^{\hat\alpha}(t)$ $\PP$-a.s. for every $t\in [0,T]$.
	Since $\hat\alpha_n\to \alpha$, we therefore conclude that 
	\begin{equation*}
		\Big\{ \partial_{\alpha}f(t, X^{\hat\alpha}(t), \hat\alpha(t)) + \partial_{\alpha}b_2\big(t, X^{\hat\alpha}(t),  \hat\alpha(t) \big)Y^{\hat\alpha}(t) \Big\}\cdot(\beta - \hat\alpha(t)) \ge 0 \quad \PP\otimes \mathrm{d}t\text{-a.s.}
	\end{equation*}
	This shows \eqref{eq:nec.cond}, which concludes the proof.
\end{proof}		
	
	\section{The sufficient condition for optimality}
	\label{sec:sufficient}

Let us now turn to the proof of the sufficient condition of optimality.
Since we will need to preserve the concavity of $H$ assumed in Theorem \ref{thm:suff} after approximation, we specifically assume that the function $b_n$ is defined by standard mollification.
Therefore, $H_n(t,x,y,a):= f(t,x,a)+ b_n(t,x,a)y$ is a mollification of $H$ and thus remains concave.
\begin{proof}(of Theorem \ref{thm:suff})
	Let $\hata \in \mathcal{A}$ satisfy \eqref{eq:suff.con} and $\alpha'$ an arbitrary element of $\mathcal{A}$.
	We would like to show that $J(\hata) \ge J(\alpha')$.
	Let $n \in \mathbb{N}$ be arbitrarily chosen.
	By definition, we have
	\begin{align*}
		&J_n(\hata) - J_n(\alpha')\\
		& = \E\Big[g(X^{\hata}_n(T)) - g(X^{\alpha'}_n(T)) + \int_0^Tf(u, X_n^{\hata}(u), \hata(u)) - f(u, X_n^{\alpha'}(u), \alpha'(u))\diff u  \Big]	\\
		&\ge \E\Big[\partial_xg(X^{\hata}_n(T))\big\{X^{\hata}(T) -X^{\alpha'}_n(T)\big\} + \int_0^T\big\{ b_n(u, X_n^{\alpha'}(u), \alpha'(u)) - b_n(u, X_n^{\hata}(u),\hata(u))\big\} Y_n^{\hata}(u)\diff u\\
		&\quad + \int_0^T H_n(u, X_n^{\hata}(u), Y_n^{\hata}(u), \hata(u)) - H_n(u, X_n^{\alpha'}(u),Y_n^{\hata}(u), \alpha'(u))\diff u  \Big],
	\end{align*}	
	where we used the definition of $H_n$ and the fact that $g$ is concave.
	Since $Y_n^{\hata}$ satisfies
	\begin{equation*}
		Y^{\hata}_n(t) = \E\Big[\Phi_n^{\hata}(t,T) \partial_xg( X^{\hata}_n(T)) + \int_t^T\Phi_n^{\hata}(t,u) \partial_xf(u, X_n^{\hata}(u), \hata(u))\mathrm{d}u\mid \mathcal{F}_t \Big],
	\end{equation*}
	it follows by martingale representation and It\^o's formula that there is a square integrable progressive process $(Y^{\hata}_n,Z^{\hata}_n)$  such that $Y_n^{\hata}$ satisfies the (linear) equation 
	\begin{equation*}
		Y^{\hata}_n(t) = \partial_xg(X^{\hata}_n) + \int_t^T\partial_xH_n(u, X^{\hata}_n(u), Y_n^{\hata}(u),\hata(u))\diff u - \int_t^TZ_n^{\hata}(u)\diff W(u).
	\end{equation*}
	Recall that since $b_n$ is smooth, so is $H_n$.
	Therefore, by It\^o's formula once again we have
	\begin{align*}
		&Y^{\hata}_n(T)\big\{X_n^{\hata}(T) - X_n^{\alpha'}(T)\big\} = \int_0^TY^{\hata}_n(u)\big\{b_n(u, X^{\hata}_n(u),\hata(u)) - b_n(u, X^{\alpha'}_n(u),\alpha'(u)) \big\}\diff u\\
		&\quad - \int_0^T\big\{X^{\hata}_n(u) - X^{\alpha'}_n(u) \big\}\partial_xH_n(u, X^{\hata}_n(u), Y_n^{\hata}(u),\hata(u))\diff u + \int_0^T\big\{X^{\hata}_n(u) - X^{\alpha'}_n(u) \big\} Z^{\hata}_n(u)\diff W(u).
	\end{align*}
	Since the stochastic integral above is a local martingale, a standard localization argument allows to take expectation on both sides to get that
	\begin{align*}
		J_n(\hata) - J_n(\alpha') &\ge \E\Big[- \int_0^T\big\{X^{\hata}_n(u) - X^{\alpha'}_n(u) \big\}\partial_xH_n(u, X^{\hata}_n(u), Y_n^{\hata}(u),\hata(u))\diff u \\
		&\quad + \int_0^T H_n(u, X_n^{\hata}(u), Y_n^{\hata}(u), \hata(u)) - H_n(u, X_n^{\alpha'}(u),Y_n^{\hata}(u), \alpha'(u))\diff u   \Big]\\
		&\ge \E\Big[\int_0^T \partial_\alpha H_n(u, X_n^{\hata}(u), Y_n^{\hata}(u), \hata(u))\cdot(\hata(u) - \alpha'(u))\diff u  \Big],
	\end{align*}
	where the latter inequality follows by concavity of $H_n$.
		
	Coming back to the expression of interest $J(\hata) - J(\alpha')$, we have
	\begin{align*}
		J(\hata) - J(\alpha') & = J(\hata) - J_n(\hata) + J_n(\hata) - J_n(\alpha') + J_n(\alpha') - J(\alpha')\\
		&\ge J(\hata) - J_n(\hata) + \E\Big[\int_0^T \partial_\alpha H_n(u, X_n^{\hata}(u), Y_n^{\hata}(u), \hata(u))\cdot(\hata(u) - \alpha'(u))\diff u  \Big]\\
		&\quad  + J_n(\alpha') - J(\alpha').
	\end{align*}
	Since $b_{1,n}$ does not depend on $\alpha$, we have
	$\partial_\alpha H_n(u, X_n^{\hata}(u), Y_n^{\hata}(u), \hata(u)) = \partial_\alpha b_2(u, X^{\hata}_n(u),\hata(u))Y^{\hata}_n(u) + \partial_\alpha f(u, X^{\hata}_n(u),\hata(u))$.
	Therefore, taking the limit as $n$ goes to infinity, it follows by Lemmas \ref{lem:conv.Xnn}, \ref{lem:J.continuous} and \ref{lem:conv.y.phi} that it holds
	\begin{align*}
		J(\hata) - J(\alpha') \ge E\Big[\int_0^T \partial_\alpha H(u, X^{\hata}(u), Y^{\hata}(u), \hata(u))\cdot(\hata(u) - \alpha'(u))\diff u  \Big].
	\end{align*}
	Since $\hata$ satisfies \eqref{eq:suff.con}, we therefore conclude that $J(\hata) \ge J(\alpha')$.
\end{proof}
	
	\subsection{Concluding remarks}
	\label{subsec.conclusion}
	
Let us conclude the paper by briefly discussing our assumptions.
The condition $b=b_1+b_2$ seems essential to derive existence and uniqueness results of the controlled system. 
For instance, the crucial bound \eqref{eq:bound.bprime} derived in \cite{BMBPD17,MMNPZ13} is unknown when $b_1$ depends on $\alpha$.
This condition is also vital in obtaining the explicit representation of the Sobolev derivative of the flows of the solution to the SDE in terms of its local time. 
This representation cannot be expected in multidimensions due to the non commutativity of matrices and the local time.
Therefore, much stronger (regularity) conditions are needed to derive the maximum principle in this case (see for example \cite{Bah-Chi-Dje-Mer, Bah-Dje-Mer-AMO07, Bah-Dje-Mer07}). 
Note in addition that the boundedness assumption on $b$ is made mostly to simplify the presentation.
The results should also hold with $b$ of linear growth in the spacial variable, albeit with more involved computations and with $T$ small enough, since the flow in this case is expected to exist in small time.
	
Given the drift $b$, some known conditions on the control $\alpha$ that guaranty existence and uniqueness of the strong solution to the SDE \eqref{eqSpro1} satisfied by the controlled process are given by \eqref{eqcondal1} and \eqref{eqcondal2}.
These conditions involve the Malliavin derivative of $\alpha$. 
Let us remark that the Malliavin differentiability of the control is not an uncommon assumption. This condition appears implicitly in the works \cite{Menou20142, MOZ12, OS09} on the stochastic maximum principle where the coefficients are required to be at least two times differentiable with bounded derivatives. 

\begin{appendix}

	\section{Representation of the differential flow by time-space local time}
		
	It is well-known that solutions of stochastic differential equations admit a stochastic differential flow.
	Such flows have been extensively investigated in the work of Kunita \cite{Kun90} for equations with sufficiently smooth coefficients.
	When the drift merely measurable, it turns out (see e.g. \cite{MMNPZ13,MNP2015,XichZhang16}) that flows still exists, at least in the Sobolev sense.
	The study of existence of such flows is extended to the case of random coefficients in \cite{MenTan19}.
	In this appendix, we show that the stochastic differential flow admits an explicit representation.
	The difficulty here is the lack of regularity of the drift, around which we get using local time integration. 
	This representation has been obtained in \cite{BMBPD17} assuming that the drift $b=b_1+b_2$ is deterministic with $b_1$ bounded and measurable and $b_2$ Lipschitz--continuous. 
		
	\begin{thm}\label{Thmexpliflowder}
		Suppose that $b$ is as in Theorem \ref{thm:necc} and $\alpha \in \mathcal{A}$. 
		For every $0\leq s\leq t\leq T$, the stochastic flows $\Phi^{\alpha, x}(t,s)$ of the unique strong solution to the SDE \eqref{eqSpro1} admits the representation
		\begin{align}\label{eqflow11}
			\Phi^{\alpha,x}(t,s)=&\exp\Big(-\int_s^{t}\int_{\mathbb{R}}b_1\left(u,z\right)L^{X^{\alpha,x}}(\mathrm{d}u,\mathrm{d}z)+\int_s^{t}b'_2\left(u,X^{\alpha,x}(u),\alpha(u)\right) \mathrm{d}u\Big).
		\end{align}
		Here $\int_s^t\int_{\mathbb{R}}b_1(u,z)L^{X^x}(\diffns u,\diffns z)$ is the integration with respect to the time-space local time of $X^x$ and $b'_2$ is the derivative with respect to the second parameter. 
	\end{thm}	
		
	\begin{proof} 
		We know from \cite{MenTan19}, \cite{BMBPD17} that under the condition of the Theorem, the SDE \eqref{eqSpro1} has a Sobolev differentiable flow denoted $\Phi^{\alpha,x}$.
		In particular, it is shown in these references that $\Phi^{\alpha,x}_n(t,s)$ converges to $\Phi^{\alpha,x}(t,s)$ weakly in $L^2(U\times\Omega)$.

		Thus, in order to show the representation \eqref{eqflow11}, it suffices to show that $\Phi^{\alpha,x}_n(t,s)$ converges to 
		$$
			\Gamma^{\alpha,x}(t,s):=e^{\int_s^{t}\int_{\mathbb{R}}b_1\left(u,z\right)L^{X^{\alpha,x}}(\mathrm{d}u,\mathrm{d}z)}e^{\int_s^{t}b'_2\left(u,X^{\alpha,x}(u),\alpha(u)\right) \mathrm{d}u}
		$$ 
		weakly in $L^2(U\times\Omega)$. 
		Since the set
		$$
			\Big\{h\otimes \mathcal{E}\Big(\int_0^1\dot{\varphi}(u)\mathrm{d}B(u)\Big):\varphi\in C^{1}_b(\mathbb{R}),h\in C^\infty_0(U)\Big\}
		$$
		spans a dense subspace in $L^2(U\times\Omega)$, it is therefore enough to show that 
		$$
			\int_{\mathbb{R}}h(x)E\Big[\Phi^{\alpha,x}_n(t,s) \mathcal{E}\Big(\int_0^1\dot{\varphi}(u)\mathrm{d}B(u)\Big)\Big]\mathrm{d}x\rightarrow \int_{\mathbb{R}}h(x)E\Big[\Gamma^{\alpha,x}(t,s) \mathcal{E}\Big(\int_0^1\dot{\varphi}(u)\mathrm{d}B(u)\Big)\Big]\mathrm{d}x.
		$$
		Recall that for $\varphi\in C^1_b([0,T],\mathbb{R}^d)$, for every $n$, the process $\tilde X^{\tilde \alpha,x}_n:=X^{\tilde \alpha,x}_n(\omega+\varphi)$, with $\tilde\alpha(\omega)=\alpha(\omega+\varphi)$ satisfies the SDE
		\begin{align}\label{eqxntilde1}
			\diffns \tilde X^{\tilde \alpha,x}_n(t)=(b_{1,n}(t,\tilde X^{\tilde \alpha,x}_n(t))+ b_{2}(t,\tilde X^{\tilde \alpha,x}_n(t),\tilde \alpha)+\sigma \dot\varphi)\diffns t+\sigma \diffns B(t).
		\end{align} 
			
		We have by using Cameron-Martin theorem, the fact that $|e^x-e^y|\leq |x-y||e^x+e^y|$, H\"older inequality and boundedness of $b_2^\prime$ that
		\begin{align*}
			&\Big|\int_{\mathbb{R}}h(x)\E\Big[\Phi^{\alpha,x}_n(t,s) \mathcal{E}\Big(\int_0^1\dot{\varphi}(u)\mathrm{d}B(u)\Big)\Big]\mathrm{d}x - \int_{\mathbb{R}}h(x)\E\Big[\Gamma^{\alpha,x}(t,s) \mathcal{E}\Big(\int_0^1\dot{\varphi}(u)\mathrm{d}B(u)\Big)\Big]\mathrm{d}x\Big|\notag\\
			=&\Big|\int_{\mathbb{R}}h(x)\E\Big[
			e^{\int_s^{t}\int_{\mathbb{R}}b_{1,n}\left(u,z\right)L^{X^{\alpha,x}_n}(\mathrm{d}u,\mathrm{d}z)}e^{\int_s^{t}b'_2\left(u,X^{\alpha,x}_n(u),\alpha(u)\right)\diffns u} \mathcal{E}\Big(\int_0^1\dot{\varphi}(u)\mathrm{d}B(u)\Big)\Big]\mathrm{d}x\\
			&-\int_{\mathbb{R}}h(x)\E\Big[e^{\int_s^{t}\int_{\mathbb{R}}b_1\left(u,z\right)L^{X^{\alpha,x}}(\mathrm{d}u,\mathrm{d}z)}e^{\int_s^{t}b'_2\left(u,X^{\alpha,x}(u),\alpha(u)\right)\diffns u} \mathcal{E}\Big(\int_0^1\dot{\varphi}(u)\mathrm{d}B(u)\Big)\Big]\mathrm{d}x\Big|\notag\\
			=&\Big|\int_{\mathbb{R}}h(x)\E\Big[
			e^{\int_s^{t}\int_{\mathbb{R}}b_{1,n}\left(u,z\right)L^{\tilde X_n^{\tilde \alpha,x}}(\mathrm{d}u,\mathrm{d}z)}e^{\int_s^{t}b'_2\left(u,\hat X_n^{\tilde \alpha,x}(u),\tilde \alpha(u)\right)\diffns u} \Big]\mathrm{d}x\\
			&-\int_{\mathbb{R}}h(x)\E\Big[e^{\int_s^{t}\int_{\mathbb{R}}b_1\left(u,z\right)L^{\tilde X^{\tilde \alpha,x}}(\mathrm{d}u,\mathrm{d}z)}e^{\int_s^{t}b'_2\left(u,\tilde X^{\tilde \alpha,x}(u),\tilde \alpha(u)\right)\diffns u} \Big]\mathrm{d}x\Big|\\
				=&\Big|\int_{\mathbb{R}}h(x)\E\Big[
			e^{\int_s^{t}\int_{\mathbb{R}}b_{1,n}\left(u,z\right)L^{\tilde X_n^{\tilde \alpha,x}}(\mathrm{d}u,\mathrm{d}z)}\Big(e^{\int_s^{t}b'_2\left(u,\tilde X_n^{\tilde \alpha,x}(u),\tilde \alpha(u)\right)\diffns u}-e^{\int_s^{t}b'_2\left(u,\tilde X^{\tilde \alpha,x}(u),\tilde \alpha(u)\right)\diffns u}\Big) \Big]\mathrm{d}x\\
			&+\int_{\mathbb{R}}h(x)\E\Big[\Big(e^{\int_s^{t}\int_{\mathbb{R}}b_{1,n}\left(u,z\right)L^{\tilde X_n^{\tilde \alpha,x}}(\mathrm{d}u,\mathrm{d}z)}-e^{\int_s^{t}\int_{\mathbb{R}}b_1\left(u,z\right)L^{\tilde X^{\tilde \alpha,x}}(\mathrm{d}u,\mathrm{d}z)}\Big)e^{\int_s^{t}b'_2\left(u,\tilde X^{\tilde\alpha,x}(u),\tilde \alpha(u)\right)\diffns u} \Big]\mathrm{d}x\Big|\\
			\leq &\int_{\mathbb{R}}|h(x)|\E\Big[
			e^{2\int_s^{t}\int_{\mathbb{R}}b_{1,n}\left(u,z\right)L^{\tilde X_n^{\tilde \alpha,x}}(\mathrm{d}u,\mathrm{d}z)}\Big]^{\frac{1}{2}}\E\Big|e^{\int_s^{t}b'_2\left(u,\tilde X_n^{\tilde \alpha,x}(u),\tilde \alpha(u)\right)\diffns u}-e^{\int_s^{t}b'_2\left(u,\tilde X^{\tilde \alpha,x}(u),\tilde \alpha(u)\right)\diffns u}\Big|^2 \Big]^{\frac{1}{2}}\mathrm{d}x\\
			&+C\int_{\mathbb{R}}|h(x)|\E\Big[\Big|e^{\int_s^{t}\int_{\mathbb{R}}b_{1,n}\left(u,z\right)L^{\tilde X_n^{\tilde \alpha,x}}(\mathrm{d}u,\mathrm{d}z)}-e^{\int_s^{t}\int_{\mathbb{R}}b_1\left(u,z\right)L^{\tilde X^{\tilde \alpha,x}}(\mathrm{d}u,\mathrm{d}z)}\Big|^2 \Big]^{\frac{1}{2}}\E\Big[e^{2\int_s^{t}b'_2\left(u,\tilde X^{\tilde \alpha,x}(u),\tilde \alpha(u)\right)\diffns u} \Big]^{\frac{1}{2}}\mathrm{d}x\\		
			\leq &C\int_{\mathbb{R}}|h(x)|\Big\{\E\Big[
			e^{2\int_s^{t}\int_{\mathbb{R}}b_{1,n}\left(u,z\right)L^{\tilde X_n^{\tilde \alpha,x}}(\mathrm{d}u,\mathrm{d}z)}\Big]^{\frac{1}{2}}\int_s^{t}\E\Big[\Big|b'_2\left(u,\tilde X_n^{\tilde \alpha,x}(u),\tilde \alpha(u)\right)-b'_2\left(u,\tilde X^{\tilde \alpha,x}(u),\tilde \alpha(s)\right)\Big|^2 \Big]^{\frac{1}{4}}\diffns s\Big\}\diffns x\\
			&+C\int_{\mathbb{R}}|h(x)|\E\Big[\Big|e^{\int_s^{t}\int_{\mathbb{R}}b_{1,n}\left(u,z\right)L^{\tilde X_n^{\tilde \alpha,x}}(\mathrm{d}u,\mathrm{d}z)}-e^{\int_s^{t}\int_{\mathbb{R}}b_1\left(u,z\right)L^{\tilde X^{\tilde \alpha,x}}(\mathrm{d}u,\mathrm{d}z)}\Big|^2 \Big]^{\frac{1}{2}}
			\mathrm{d}x,
		\end{align*}
		where the last inequality follows from the boundedness of $b_2$ and $b'_2$.
		By Lemma \ref{lem:bound.int.local.time}, we have that $\E[ e^{2\int_s^{t}\int_{\mathbb{R}}b_{1,n}\left(u,z\right)L^{\tilde X_n^{\tilde \alpha,x}}(\mathrm{d}u,\mathrm{d}z)}]$ is bounded. 
		The second term on the right side of the above converges to zero since one can show as in Lemma \ref{lem:conv.Xnn} that $\tilde X^{n,\tilde \alpha,x}(s)$ converges strongly to $\tilde X^{\tilde \alpha,x}(s)$ in $L^2$ and $b_2^\prime$ is bounded and continuous. 
		
		We now show that the second term converges to zero. 
		We will show weak convergence and convergence in mean square. 
		Using the Cameron-Martin-Girsanov theorem as above, for every $\varphi_1 \in C^1_b([0,T],\mathbb{R}^d)$ we have
		\begin{align}
			&\Big|\E\Big[\mathcal{E}\Big(\int_0^T\dot{\varphi_1}(v)\mathrm{d}B(v)\Big)\Big\{e^{\int_s^{t}\int_{\mathbb{R}}b_{1,n}\left(v,z\right)L^{\tilde X_n^{\tilde \alpha,x}}(\mathrm{d}v,\mathrm{d}z)}-e^{\int_s^{t}\int_{\mathbb{R}}b_1\left(v,z\right)L^{\tilde X^{\tilde \alpha,x}}(\mathrm{d}v,\mathrm{d}z)}\Big\}\Big]\Big|\notag\\
			=&\Big|\E\Big[e^{\int_s^{t}\int_{\mathbb{R}}b_{1,n}\left(v,z\right)L^{\tilde {\tilde X}_n^{\tilde {\tilde \alpha},x}}(\mathrm{d}v,\mathrm{d}z)}-e^{\int_s^{t}\int_{\mathbb{R}}b_1\left(v,z\right)L^{\tilde{\tilde  X}^{\tilde {\tilde \alpha},x}}(\mathrm{d}v,\mathrm{d}z)}\Big]\Big|\notag\\
			=&\Big|\E\Big[\mathcal{E}\Big(\int_0^T\{u_n(v,x+\sigma\cdot B(v),\alpha (v,\omega+\varphi+\varphi_1))+\sigma\cdot(\dot{\varphi}(v)+\dot{\varphi_1}(v))\}\mathrm{d}B(v)\Big)e^{\int_s^{t}\int_{\mathbb{R}}b_{1,n}\left(v,z\right)L^{\|\sigma\|B^x_\sigma}(\mathrm{d}v,\mathrm{d}z)}\notag\\
			&-\mathcal{E}\Big(\int_0^T\{u(v,x+\sigma\cdot B(v),\alpha (v,\omega+\varphi+\varphi_1))+\sigma\cdot(\dot{\varphi}(v)+\dot{\varphi_1}(v))\}\mathrm{d}B(v)\Big)e^{\int_s^{t}\int_{\mathbb{R}}b_1\left(v,z\right)L^{\|\sigma\|B^x_\sigma}(\mathrm{d}v,\mathrm{d}z)}\Big]\Big|.
		\end{align}
		Therefore, using the inequality $|e^x-e^y|\leq |x-y||e^x+e^y|$ and H\"older's inequality we have
		\begin{align}	
			\leq &	\Big|\E\Big[\mathcal{E}\Big(\int_0^T\{u_n(v,x+\sigma\cdot B(v),\alpha (v,\omega+\varphi+\varphi_1))+\sigma\cdot(\dot{\varphi}(v)+\dot{\varphi_1}(v))\}\mathrm{d}B(v)\Big)\notag\\
			&\times\Big|\int_s^{t}\int_{\mathbb{R}}b_{1,n}\left(v,z\right)L^{\|\sigma\|B^x_\sigma}(\mathrm{d}v,\mathrm{d}z)-\int_s^{t}\int_{\mathbb{R}}b_1\left(v,z\right)L^{\|\sigma\|B^x_\sigma}(\mathrm{d}v,\mathrm{d}z)\Big|\notag\\
			&\times \Big(e^{\int_s^{t}\int_{\mathbb{R}}b_{1,n}\left(v,z\right)L^{\|\sigma\|B^x_\sigma}(\mathrm{d}v,\mathrm{d}z)}+e^{\int_s^{t}\int_{\mathbb{R}}b_1\left(v,z\right)L^{\|\sigma\|B^x_\sigma}(\mathrm{d}v,\mathrm{d}z)}\Big)\Big]\Big|\notag\\
			&+\Big|E\Big[e^{\int_s^{t}\int_{\mathbb{R}}b_1\left(v,z\right)L^{\|\sigma\|B^x_\sigma}(\mathrm{d}v,\mathrm{d}z)}\notag\\
			&\times \Big\{\mathcal{E}\Big(\int_0^T\{u_n(v,x+\sigma\cdot B(v),\alpha (v,\omega+\varphi+\varphi_1))+\sigma\cdot(\dot{\varphi}(v)+\dot{\varphi_1}(v))\}\mathrm{d}B(v)\Big)\notag\\
			&-\mathcal{E}\Big(\int_0^T\{u(v,x+\sigma\cdot B(v),\alpha (v,\omega+\varphi+\varphi_1))+\sigma\cdot(\dot{\varphi}(v)+\dot{\varphi_1}(v))\}\mathrm{d}B(v)\Big)\Big\}\Big]\Big|\notag\\
			\leq &	4\E\Big[\mathcal{E}\Big(\int_0^T\{u_n(v,x+\sigma\cdot B(v),\alpha (v,\omega+\varphi+\varphi_1))+\sigma\cdot(\dot{\varphi}(v)+\dot{\varphi_1}(v))\}\mathrm{d}B(v)\Big)^4\Big]^{\frac{1}{4}}\notag\\
			&\times \E\Big[\Big|\int_s^{t}\int_{\mathbb{R}}\Big(b_{1,n}\left(v,z\right)-b_1\left(v,z\right)\Big)L^{\|\sigma\|B^x_\sigma}(\mathrm{d}v,\mathrm{d}z)\Big|^2\Big]^{\frac{1}{2}} \notag\\
			&\times \E\Big[e^{4\int_s^{t}\int_{\mathbb{R}}b_{1,n}\left(v,z\right)L^{\|\sigma\|B^x_\sigma}(\mathrm{d}v,\mathrm{d}z)}+e^{4\int_s^{t}\int_{\mathbb{R}}b_1\left(v,z\right)L^{\|\sigma\|B^x_\sigma}(\mathrm{d}v,\mathrm{d}z)}\Big]^{\frac{1}{4}}\notag\\
			&+\E\Big[e^{2\int_s^{t}\int_{\mathbb{R}}b_1\left(v,z\right)L^{\|\sigma\|B^x_\sigma}(\mathrm{d}v,\mathrm{d}z)}\Big]^{\frac{1}{2}}\notag\\
			&\times\E\Big[\Big\{\mathcal{E}\Big(\int_0^T\{u_n(v,x+\sigma\cdot B(v),\alpha (v,\omega+\varphi+\varphi_1))+\sigma\cdot(\dot{\varphi}(v)+\dot{\varphi_1}(v))\}\mathrm{d}B(v)\Big)\notag\\
			&-\mathcal{E}\Big(\int_0^T\{u
			(v,x+\sigma\cdot B(v),\alpha (v,\omega+\varphi+\varphi_1))+\sigma\cdot(\dot{\varphi}(v)+\dot{\varphi_1}(v))\}\mathrm{d}B(v)\Big)\Big\}^2\Big]^{\frac{1}{2}}\notag\\
			=&J_{1,n}^{\frac{1}{4}}\times J_{2,n}^{\frac{1}{2}}\times J_{3,n}^{\frac{1}{4}}+J_{4,n}^{\frac{1}{2}}\times J_{5,n}^{\frac{1}{2}}.
		\end{align}
		Lemma \ref{Lemmbound1}, shows that $J_{2,n}$ converges to zero, and convergence to zero of $J_{5,n}$ follows by dominated convergence. 
		Thanks to Lemma \ref{lemmaexpoloc} and boundedness of $b_{1,n}$ and $b_1$, respectively, the term $J_{3,n}$ (respectively $J_{4,n}$) is bounded. The bound of $J_{1,n}$ follows by the uniform boundedness of $u_n$. 	
			
		Set $A_n^{\alpha}(t)=e^{\int_s^{t}\int_{\mathbb{R}}b_{1,n}\left(u,z\right)L^{\tilde X_n^{\tilde \alpha,x}}(\mathrm{d}u,\mathrm{d}z)}$ and $A^{\alpha}(t)=e^{\int_s^{t}\int_{\mathbb{R}}b_1\left(u,z\right)L^{\tilde X^{\tilde \alpha,x}}(\mathrm{d}u,\mathrm{d}z)}$.
		It remains to show convergence of the second moment, i.e. that $\E[|A_n^{\alpha}(t)|^2]$ converges to $\E[|A^{\alpha}(t)|^2]$ in $\mathbb{R}$. This follows as in the proof of Lemma \ref{lem:conv.y.phi}.
		The desired result follows. 
	\end{proof}
	We know from \cite[Theorem 2.1]{Ein2006} that the local time-space integral of $f \in {\mathcal H}^0$ admits the decomposition 
	\begin{align}\label{eqslocalt1}
	&\int_0^t\int_{\mathbb{R}}f(s,z) L^{B_a^x}(\diffns s,\diffns z)\notag\\
	=&a\int_0^t f (s,B_a^{x}(s))\diffns B(s)+a\int_{T-t}^T f (T-s,\widehat{B}_a^{x}(s))\diffns W(s)-a\int_{T-t}^T f (T-s,\widehat{B}_a^x(s))\frac{\widehat{B}(s)}{T-s}\diffns s,
	\end{align}
	$0\leq t\leq T$, a.s.,  where $\widehat{B}$ is the time-reversed Brownian motion, that is
	\begin{align}\label{eqstimrevbm1}
		\widehat{B}(t):=B(T-t),\,\,0\leq t\leq T.
	\end{align}
	In addition, the process $W=\{W(t),\,\,\,0\leq t\leq T\}$ is an independent Brownian motion with respect to the filtration $\mathcal{F}_t^{\widehat{B}}$ generated by $\widehat{B}_t$, and satisfies:
	\begin{align}\label{eqstimrevbm2}
		W(t)= \widehat{B} (t)-B(T)+\int_t^T\frac{\widehat{B}(s)}{T-s}\diffns s.
	\end{align}
	\begin{lemm}\label{Lemmbound1}
		Let $\varphi\in C^1_b([0,T],\mathbb{R}^d)$ and define $F_{1,n}$ and $F_{2,n}$ by 
		\begin{align}
			F_{1,n}:=&\int_s^{t}\int_{\mathbb{R}}\Big(b_{1,n}(u,z)-b_1(u,z)\Big)L^{\|\sigma\|B^x_\sigma}(\mathrm{d}u,\mathrm{d}z),\label{eqF1n}
		\end{align}
		Then $\E[|F_{1,n}|^2]$ converges to zero as $n$ goes to $\infty$. 
	\end{lemm}
	\begin{proof}
		Using the local time-space decomposition \eqref{eqslocalt1}, the Minkowski integral inequality with the measure $\nu(\sigma)=\int_{\sigma}\frac{\diffns s}{2\sqrt{T-s}}$, the H\"older and the Burkholder-Davis-Gundy inequalities, we get
		\begin{align*}
	 		\E[|F_{1,n}|^2] 
			\leq &4\|\sigma\|^2\E\Big[\Big\{\int_t^s \Big(b_{1,n} (u,B^{x}_\sigma(u))-b_{1} (u,B_\sigma^{x}(u))\Big)\diffns B(s)\Big\}^2\Big]\\
			&	+4\E\Big[\Big\{\int_{T-t}^{T-s} \Big(b_{1,n}(T-u,\widehat{B}_\sigma^{x}(u))-b_{1}(T-u,\widehat{B}_\sigma^{x}(u))\Big)\diffns W(u)\Big\}^2\Big]\\
			&+4\E\Big[\Big\{\int_{T-t}^{T-s} \Big(b_{1,n}(T-u,\widehat{B}_\sigma^x(u))-b_{1}(T-u,\widehat{B}_\sigma^x(u))\Big)\frac{\widehat{B}(u)}{\sqrt{T-u}}\frac{\diffns u}{\sqrt{T-u}}\Big\}^2\Big]\\
			\leq &C_\sigma\Big\{ \int_t^s \E\Big[\big| b_{1,n} (u,B_\sigma^{x}(u))-b_{1} (u,B_\sigma^{x}(u))\big|^2\Big]\diffns u\\
			&	+\int_{T-t}^{T-s} \E\Big[\big|b_{1,n}(T-u,\widehat{B}^{x}_\sigma(u))-b_{1}(T-u,\widehat{B}_\sigma^{x}(u))\big|^2\Big]\diffns u\\
			&+\Big(\int_{T-t}^{T-s} \E\Big[\Big(b_{1,n}(T-u,\widehat{B}_\sigma^x(u))-b_{1}(T-u,\widehat{B}_\sigma^x(u))\Big)^2\Big(\frac{\widehat{B}(u)}{\sqrt{T-u}}\Big)^2\Big]^{\frac{1}{2}}\frac{\diffns s}{\sqrt{T-u}}\Big)^2\Big\}.
		\end{align*}
		Now using the Cauchy-Schwartz inequality and the fact that $E[B^4(t)]=3t^2$, we can continue the estimation as
		\begin{align*}
			\E[|F_{1,n}|^2] \leq &C_\sigma\Big\{ \int_t^s \E\Big[\big|b_{1,n} (u,B_\sigma^{x}(u))-b_{1} (u,B_\sigma^{x}(u))\big|^2\Big]\diffns u\\
			&	+\int_{T-t}^{T-s} \E\Big[\big|b_{1,n}(T-u,\widehat{B}^{x}_\sigma(u))-b_{1}(T-u,\widehat{B}_\sigma^{x}(u))\big|^2\Big]\diffns u\\
			&+\Big(\int_{T-t}^{T-s} \E\Big[\big|b_{1,n}(T-u,\widehat{B}_\sigma^x(u))-b_{1}(T-u,\widehat{B}_\sigma^x(u))\big|^4\Big]^{\frac{1}{4}}\frac{\diffns s}{\sqrt{T-u}}\Big)^2\Big\}.
		\end{align*}

		Each term above converges to zero.
		We give the detail only for the first term.
		The treatment of the two oder terms is analogous.
		Given $p>1$, using the density of the Brownian motion, we have as in the proof of Lemma \ref{lem:conv.Xnn} (see \eqref{eq:estim.bnb})
		\begin{align*}
			\E\Big[\big|b_{1,n} (s,B^{x}(s))-b_{1} (s,B^{x}(s))\big|^p\Big]
			\leq &\frac{1}{\sqrt{2\pi s}}e^{\frac{x^2}{2s}}\int_{\mathbb{R}}\big|b_{1,n} (s,y)-b_{1} (s,y)\big|^pe^{-\frac{y^2}{4s}}\diffns y.
		\end{align*}
		Since $b_{1,n}$ converges to $b_1$, it follows from the dominated convergence theorem that each term in the above inequality converge to zero. 
	\end{proof}
		
	The following Lemma corresponds to \cite[Lemma A.2]{BMBPD17} and it gives the exponential bound of the local time-space integral of a bounded function
		
	\begin{lemm}\label{lemmaexpoloc}
		Let $b:[0,T]\times \mathbb{R} \rightarrow \mathbb{R}$ be a bounded and measurable function. Then for $t\in [0,T],\, \lambda \in  \mathbb{R}$ and compact subset $K\subset \mathbb{R}$, we have
		$$
			\underset{x\in K}{\sup} \E\Big[\exp\Big(\lambda \int_0^t\partial_xb(s,B^x)\diffns s\Big) \Big]=\underset{x\in K}{\sup} \E\Big[\exp\Big(\lambda \int_0^t\int_{\mathbb{R}}b(s,y)L^{B^x}(\diffns s,\diffns y)\Big) \Big]<C(\|b\|_{\infty}),
		$$
		where $C$ is an increasing function and $L^{B^x}(\diffns s,\diffns y)$ denotes integration with respect to the local time of the Brownian motion $B^x$ in both time and space. In addition, if $b_n$ is an approximating sequence of $b$ such that the $b_n$ are uniformy bounded by $\|b\|_{\infty}$ then the above bound still hold true with the bound independent of $n$.
		\end{lemm}
\end{appendix}

\bibliographystyle{abbrv}
	

\vspace{.3cm}

\noindent Olivier Menoukeu-Pamen: University of Liverpool Institute for Financial and Actuarial Mathematics, Department of Mathematical Sciences,
L69 7ZL, United Kingdom and African Institute for Mathematical Sciences, Ghana. menoukeu@liverpool.ac.uk\\
Financial support from the Alexander von
Humboldt Foundation, under the program financed by the German Federal Ministry of Education and Research
entitled German Research Chair No 01DG15010 is gratefully acknowledged.
  \vspace{.2cm}

\noindent Ludovic Tangpi: Department of Operations Research and Financial Engineering, Princeton University, Princeton, 08540,
 NJ; USA. ludovic.tangpi@princeton.edu\\
Financial suupport by NSF grant DMS-2005832 is gratefully acknowledged.

\end{document}